\documentclass[a4paper, 11pt, twoside, leqno]{article}
\usepackage[T1]{fontenc}
\usepackage[utf8x]{inputenc}
\usepackage[english]{babel}
\usepackage{lmodern}               % Latin Modern font
\usepackage{amsmath,amsthm,amssymb}
\usepackage{mathrsfs,esint,bbm,stmaryrd}
\usepackage[shortlabels]{enumitem}
\usepackage{graphicx,xcolor,subcaption}
\usepackage[textwidth=16.1cm,centering,headheight=24pt]{geometry}
\usepackage{authblk}
\usepackage{mathtools}

\newtheorem{theorem}{Theorem}           % Bold title, italic text
\newtheorem{corollary}[theorem]{Corollary}
\newtheorem{lemma}[theorem]{Lemma}
\newtheorem{prop}[theorem]{Proposition}

\theoremstyle{definition}              % Bold title, roman text

\theoremstyle{remark}                  % Italic title, roman text
\newtheorem{step}{Step}

\newtheorem{remark}{Remark}

\DeclareMathOperator{\dist}{dist}                                   % distance
\DeclareMathOperator{\spt}{spt}                                     % index
\let\div\relax
\DeclareMathOperator{\div}{div}                                     % divergence

\newcommand{\abs}[1]{\left| #1 \right|}                             % absolute value
\newcommand{\norm}[1]{\left\| #1 \right\|}                          % norm
\newcommand{\csubset}{\subset\!\subset}                             % compact inclusion

\DeclareMathAlphabet{\mathpzc}{OT1}{pzc}{m}{it}

\newcommand{\T}{\mathrm{T}}
\renewcommand{\d}{\mathrm{d}}
\renewcommand{\o}{\mathrm{o}}

\newcommand{\R}{\mathbb{R}}

\newcommand{\Q}{\mathbf{Q}}     
\newcommand{\M}{\mathbf{M}}
\newcommand{\I}{\mathbf{I}}
\newcommand{\X}{\mathbf{X}}

\newcommand{\n}{\mathbf{n}}
\newcommand{\m}{\mathbf{m}}

\newcommand{\e}{\mathbf{e}}

\renewcommand{\u}{\mathbf{u}}
\newcommand{\F}{\mathscr{F}}
\renewcommand{\H}{\mathscr{H}}

\newcommand{\eps}{\varepsilon}

\newcommand{\Cpot}{{\rm C}_{\rm pot}}

\newcommand{\bbV}{\mathbb{V}}

\SetSymbolFont{stmry}{bold}{U}{stmry}{m}{n}  % Bold Greek letters
\newcommand{\nnu}{{\boldsymbol{\nu}}}

\newcommand{\ttau}{{\boldsymbol{\tau}}}

\newcommand{\Sz}{\mathscr{S}_0^{2\times 2}}
\newcommand{\Qb}{\Q_{\mathrm{bd}}}
\newcommand{\Mb}{\M_{\mathrm{bd}}}
\newcommand{\nb}{\n_{\mathrm{bd}}}
\newcommand{\bbS}{\mathbb{S}}

\renewcommand{\v}{\mathbf{v}}

\newcommand{\mres}{\mathbin{\vrule height 1.6ex depth 0pt width
0.13ex\vrule height 0.13ex depth 0pt width 1.3ex}}

\definecolor{lightblue}{rgb}{0.22,0.45,0.70}   % light blue
\definecolor{darkgray}{gray}{0.4}    % dark grey
\definecolor{lightgray}{gray}{0.8}

\title{Ferronematics: integrality of the limiting interface in the strong coupling regime}
\author{Federico Luigi Dipasquale, Yoshihiro Tonegawa}
\date{}

\newcommand{\Addresses}{{% additional braces for segregating \footnotesize
  \bigskip
  \footnotesize

  Federico~Luigi~Dipasquale \\
  \textsc{Scuola Superiore Meridionale}\\
  Via Mezzocannone 4, 80138 Napoli, Italy\\
  \textit{E-mail address}: \texttt{f.dipasquale@ssmeridionale.it}

    \medskip

   Yoshihiro~Tonegawa \\
    \textsc{Department of Mathematics, Institute of Science Tokyo}\\
    2-12-1 Ookayama, Meguro-ku, Tokyo 152-8551, Japan\\ %\par\nopagebreak
  \textit{E-mail address}: \texttt{tonegawa@math.titech.ac.jp}
}}

\begin{document}

\maketitle

\begin{abstract}
	We consider a vectorial energy functional proposed in the physical 
	literature as a simplified model for thin films of ferronematics --- composite 
	materials formed by dispersing magnetic nanoparticles into a 
	nematic liquid crystals host. The model features two order parameters: 
	a reduced $\Q$-tensor field $\Q$, describing the nematic liquid crystal component, 
	and a vector field $\M$, accounting for the average polarisation vector 
	field generated by the included particles. The energy contains a 
	coupling term promoting alignment between $\Q$ and $\M$.
	It has been shown in~\cite{CDS1,CDS2} 
	that, as a small parameter $\eps$ tends to zero, the energy of critical pairs 
	$(\Q_\eps,\,\M_\eps)$ concentrates on \emph{distinct} singular sets: a finite 
	set of points for the $\Q_\eps$-component and the support of a $\H^1$-rectifiable  
	varifold for the $\M_\eps$-component, with first variation supported on 
	the singular set for the $\Q_\eps$-component. 
	We show in this paper that, if the coupling constant is above a 
	certain (explicit) threshold, then, up to rescaling its density by a 
	material constant, such a  
	limiting varifold has integer multiplicity.
	

\end{abstract}

\section*{Introduction}
%The formation of gradient-driven singular structures has been the 
%subject of intensive investigation over the past few decades. 
%In particular, 
Over the past few decades, much effort has been devoted to the study of  
%of energy functionals like 
the Ginzburg-Landau functional of superconductivity 
and to the Allen-Cahn functional of phase separation, 
especially concerning the asymptotic behaviour, as a small 
parameter $\eps$ tends to zero, of their 
minimising and non-minimising critical points.  
%Likewise,  in the limit $\eps \to 0$ 
%has been deeply explored. 
Seminal works in these fields, such as, for instance, 
\cite{BBH, BBO, AlbertiBaldoOrlandi, JerrardSoner, SandierSerfaty, Modica, Sternberg, Ilmanen, HutchinsonTonegawa, B}, 
revealed the emergence of gradient-driven singular structures of codimension~2 
(in the case of Ginzburg-Landau theory) and codimension~1 (in the 
Allen-Cahn case) as the sets where the energy of minimisers and 
critical points concentrates, at first order, as $\eps$ tends to zero.
From the point of view of geometric measure theory, such 
energy-concentration phenomena are extremely interesting, as they provide 
a mechanism to construct non-trivial integer-multiplicity rectifiable 
currents minimising the energy in an appropriate homology class and 
stationary varifolds of codimension 1 and 2, both in the 
Euclidean and in the Riemannian context
(e.g., \cite{Baraket, Orlandi, IgnatJerrard, CDO1, CDO2, PigatiStern, ParisePigatiStern, DePhilippisPigati,  PisantePunzo1, LiPariseSarnataro}). 
Moreover, Pacard and Ritor\'e~\cite{PacardRitore} and De~Philippis and Pigati~\cite{DePhilippisPigati} have shown that 
any non-degenerate minimal submanifold of codimension 1 or 2 of a Riemannian 
manifold can be recovered as the energy-concentration set of critical 
points of the Allen-Cahn functional, if of codimension 1 (see~\cite{PacardRitore}), or of 
the Ginzburg-Landau functional, if of codimension 2 (see~\cite{DePhilippisPigati}). 

While the Ginzburg-Landau functional is inherently vectorial, 
the Allen-Cahn functional makes sense for both scalar and vectorial order parameters, 
depending on the application at hand (for instance, in the case of three or more 
phases, the vectorial theory must be used). 
The asymptotic behaviour of minimisers has been studied within the framework of 
$\Gamma$-convergence, both in the scalar and the vectorial 
case~\cite{Modica, Sternberg, FonsecaTartar, Baldo}.
%The Allen-Cahn theory is very well developed and satisfactory in the scalar case and 
%much less understood in the vectorial one, although $\Gamma$-convergence results 
%where soon obtained~\cite{Baldo, FonsecaTartar} after the publication of 
%Modica's pioneering paper~\cite{Modica}. 
%Concerning the asymptotic behaviour of (non-minimising, in general) critical points, however, the scalar and 
%the vectorial theory are completely different.  
However, when coming to the analysis of the asymptotic behaviour of 
(non-minimising, in general) critical points, 
the scalar and the vectorial theories are profoundly different.  
The asymptotic behaviour of general critical points in the scalar case was investigated 
by Hutchinson and the second author~\cite{HutchinsonTonegawa} 
through a careful PDE analysis  
%and subsequent work 
%(for instance, \cite{Tonegawa-Stability, MizunoTonegawa, NagaseTonegawa, RogerSchatzle}) 
in which a key role is played by the so-called \emph{discrepancy function},  
defined as the difference between the potential energy density and the elastic 
energy density. By cleverly exploiting the maximum principle, Modica~\cite{Modica} 
showed that the discrepancy of entire scalar solutions is pointwise non-negative, already 
at the $\eps$-level. In~\cite{HutchinsonTonegawa}, it is proved that,  
in bounded domains, the discrepancy vanishes asymptotically in the $L^1$-sense 
for every sequence of critical points with equibounded energy. In the vectorial case, 
however, there is no reason for these properties to hold (in fact,  
an example of entire solution whose discrepancy changes sign has been constructed in~\cite{Smyrnelis}) 
and the asymptotic behaviour of non-minimising critical points has eluded precise 
description up to the recent breakthrough paper~\cite{B} by Bethuel, 
in which new PDEs tools for the Allen-Cahn system on two-dimensional 
domains has been developed 
that avoid the discrepancy function and are instead ultimately based on refined 
energy estimates.

\vskip5pt

In this paper we are concerned with a model for \emph{ferronematics} which couples 
the Ginzburg-Landau theory with the vectorial Allen-Cahn theory. % in a rather non-trivial way.
Ferronematics are composite materials obtained by dispersing magnetic nanoparticles into a nematic 
liquid crystal hosts. The physics of such systems, firstly devised theoretically 
by Brochard and de~Gennes in 1970~\cite{BrochardDeGennes}, and then realised experimentally 
in various stages (see, e.g., \cite{Mertelj} and references therein) is very rich, 
%not completely understood, 
and the object of ongoing investigation. 
%From the experimental viewpoint, one of the main problems was the tendency of the 
%magnetic particles to aggregate and form chains, thus changing the physics of the 
%system. Up to recent times, very much more sophisticated experimental designed 
%were devised, this inconvenience forced experimentalists to work 
%in the so-called \emph{superdilute regime}, in which the density of magnetic particles 
%relatively low with respect to that of nematic molecules.

A simplified model for thin films of ferronematics in the so-called 
\emph{superdilute regime} (in which the density of magnetic particles 
is relatively low with respect to that of nematic molecules)
has been proposed in~\cite{Bisht2019}. The model features 
two order parameters: a reduced $\Q$-tensor field associated with the liquid crystal host, 
which is a map $\Q$ from the physical domain $\Omega \subset \R^2$ to $\Sz$, 
the space of $2 \times 2$, symmetric, real matrices with trace equal to zero,  
and the average polarisation vector field $\M : \Omega \to \R^2$ 
generated by the included particles. After non-dimensionalisation, the model can be described as follows.

Let $\Omega \subset \R^2$ be an open, bounded, and simply connected set with smooth boundary $\partial \Omega$, let $\eps > 0$ be a small parameter, and let 
$G \subseteq \Omega$ be any measurable set. We define
\begin{equation}\label{eq:functional}
    \F_\eps(\Q,\,\M;\,G) := \int_G \left\{ \frac{1}{2} \abs{\nabla \Q}^2 + \frac{\eps}{2} \abs{\nabla \M}^2 + \frac{1}{\eps^2} f_\eps(\Q,\,\M)\right\} \,{\d}x,
\end{equation}
where the potential energy density $f_\eps(\Q,\,\M)$ is given by
\begin{equation}\label{eq:f}
    f_\eps(\Q,\,\M) := \frac{1}{4}\left( \abs{\Q_\eps}^2 - 1 \right)^2 + \frac{\eps}{4}\left( \abs{\M_\eps}^2 - 1 \right)^2 - \eps \beta \Q \M \cdot \M + \chi_\eps.
\end{equation}
Here, $\beta$ is a (strictly) positive parameter depending only on the material and the additive constant $\chi_\eps$ ensures that $\inf f_\eps = 0$ for each $\eps > 0$. 

The mathematical study of the functional $\F_\eps$ 
started in \cite{CMSW} and then continued in \cite{CDS1, CDS2, CDS3}, 
encompassing the asymptotic behaviour as $\eps \to 0$ 
of both 
minimising and non-minimising critical points under suitable boundary conditions. 
The limit $\eps \to 0$ should be understood as a large domain limit 
(the size of the domain is much larger than the typical correlation length for 
the liquid crystal molecules, cf.~\cite{CMSW,CDS1}).

In this paper, we refine some of the results obtained in~\cite{CDS1,CDS2} 
for (non-minimising, in general) critical points. 
To motivate the interest towards the study of non-minimising critical points, we 
recall that, while in variational theories for physical systems observable configurations are usually identified minimisers or local minimisers 
of the energy, 
it has been argued on the basis of numerically and experimentally observed cases 
(see, e.g.,~\cite{GurtinMatano} for a broader discussion) that non-minimising critical 
points show up naturally, for instance, when the system decays from a higher to a lower energy state. 
Moreover, numerical simulations~\cite{Bisht2019,CMSW} show an abundance of critical points of the functional~\eqref{eq:functional}.

As in \cite{CDS1, CDS2}, 
we assume to deal with sequences of (non-minimising, in general) 
critical pairs $(\Q_\eps,\,\M_\eps)$ of $\F_\eps$ in $\Omega$ 
and that the boundary data are $\eps$-independent maps 
$\Qb \in C^1(\partial \Omega,\,\Sz)$, $\Mb \in C^1(\partial \Omega,\,\R^2)$ that satisfy either the `pure' Dirichlet boundary conditions 
\begin{equation}\label{hp:bc-dir}
    \abs{\Mb} = \left(\sqrt{2}\beta + 1 \right)^{1/2}, \qquad
    \Qb = \sqrt{2}\left(\frac{\Mb \otimes \Mb}{\sqrt{2}\beta + 1} - \frac{\I}{2}\right) 
\end{equation}
or the `mixed' boundary conditions 
\begin{equation}\label{hp:bc-mixed}
    \Qb = \sqrt{2}\left( \nb \otimes \nb - \frac{\I}{2} \right), \qquad 
    \partial_\nnu \M_\eps = 0,
\end{equation}
where $\nb \in C^1(\partial\Omega,\,\bbS^1)$ is an assigned, 
$\eps$-independent vector field. %with values into $\bbS^1$. 
Finally, we assume that 
%the sequence $\{(\Q_\eps,\,\M_\eps)\}_\eps$ satisfies the following: 
there exists a positive constant $\Cpot$ such that, for every $\eps > 0$, 
\begin{equation}\label{hp:equibdd-f}
    \frac{1}{\eps^2} \int_\Omega f(\Q_\eps,\,\M_\eps) \,{\d}x \leq \Cpot.
\end{equation}
Sufficient conditions for~\eqref{hp:equibdd-f} are provided in~\cite[Remark~1]{CDS1}.

%that 
%will be used in this paper as well. 
%The asymptotic analysis of minimising critica
% 
%In this paper, we study some properties of critical points of the energy functional $\F_\eps$, complementing results in~\cite{CDS}.

The Euler-Lagrange equations of the functional $\F_\eps$ are
\begin{align}
    &-\Delta \Q_\eps + \frac{1}{\eps^2} \left( \abs{\Q_\eps}^2 - 1 \right) \Q_\eps - \frac{\beta}{\eps}\left( \M_\eps \otimes \M_\eps - \frac{\abs{\M_\eps}^2}{2} \I\right) = 0, \label{eq:EL-Q} \\
    &-\Delta \M_\eps + \frac{1}{\eps^2}\left(\abs{\M_\eps}^2 - 1 \right)\M_\eps - \frac{2\beta}{\eps^2}\Q_\eps\M_\eps = 0\label{eq:EL-M}.
\end{align}
Under the above assumptions, upon using~\eqref{eq:EL-Q},~\eqref{eq:EL-M}, 
it was shown in~\cite{CDS1} that 
\begin{equation}\label{eq:energy-bounds}
    \F_\eps(\Q_\eps,\,\M_\eps) \leq C \abs{\log \eps}, \qquad
    \int_\Omega \abs{\nabla \Q_\eps}^2\,{\d}x \leq C \abs{\log \eps}, \qquad 
    \eps \int_\Omega \abs{\nabla \M_\eps}^2 \,{\d}x \leq C,
\end{equation}
where the constant $C$ does not depend on $\eps$. 

As in \cite{CDS1, CDS2}, we define the energy densities
\begin{align}
    &\mu_\eps := \frac{1}{\abs{\log\eps}} \left( \frac{1}{2} \abs{\nabla \Q_\eps}^2 + \frac{\eps}{2} \abs{\nabla \M_\eps}^2 + \frac{1}{\eps^2} f_\eps(\Q_\eps,\,\M_\eps) \right), \\
    &\nu_\eps := \frac{\eps}{2} \abs{\nabla \M_\eps}^2 + \frac{1}{\eps^2} f_\eps(\Q_\eps,\,\M_\eps).
\end{align}
%Under assumptions~\eqref{hp:equibdd-f} and boundary conditions 
%as in~\eqref{hp:bc-dir} or~\eqref{hp:bc-mixed}, 
From the assumption~\eqref{hp:equibdd-f} and the energy 
bounds~\eqref{eq:energy-bounds}, 
the sequences of positive functions $\{\mu_\eps\}_\eps$, $\{ \nu_\eps \}_\eps$ are bounded in $L^1(\Omega)$, hence we can find limiting Radon measures $\mu_\star$ and $\nu_\star$ in $\Omega$ such that, after possible extraction of a subsequence,
\[
    \mu_\eps \rightharpoonup^* \mu_\star, 
    \qquad 
    \nu_\eps \rightharpoonup^* \nu_\star
    \qquad 
    \mbox{as } \eps \to 0
\]
in the sense of Radon measures in $\Omega$. 
In the limit, \emph{distinct} (but related) gradient-driven singular 
structures show up. 
Indeed, it is shown in~\cite{CDS1} that $\spt\mu_\star$ is a finite 
set of points and in~\cite{CDS2} that $\spt\nu_\star$ is a union of 
segments with locally constant density. The relationship between 
$\spt\mu_\star$ and $\spt\nu_\star$ is expressed by item~\ref{item:balance-law} 
of Theorem~\ref{thm:B-CDS} below and, in a more precise way, by 
\cite[Theorem~C]{CDS2}.
 
These results are part of the outcome of a broader analysis on 
the systems~\eqref{eq:EL-Q},~\eqref{eq:EL-M}. 
At first sight, both~\eqref{eq:EL-Q},~\eqref{eq:EL-M} look like perturbed 
Ginzburg-Landau systems, and this is true for~\eqref{eq:EL-Q}, although 
the relatively large perturbation term makes the analysis in~\cite{CDS1} quite 
delicate. However, as shown in~\cite{CDS2}, the structure of~\eqref{eq:EL-M} is 
the one of a perturbed Allen-Cahn system.  
%
%As observed above, the structure of the system~\eqref{eq:EL-Q} is overall similar 
%to the standard Ginzburg-Landau equation, and this is key to the 
%arguments in~\cite{CDS1}, which can thus proceed in analogy with the 
%standard Ginzburg-Landau case~\cite{BBH, BBO}, although the presence 
%of the (large) coupling term makes it 
%quite non trivial and calls for additional care and some 
%delicate estimates. 
%
%The structure of the system~\eqref{eq:EL-M} is, instead, completely 
%different, and of Allen-Cahn type. In order to see this, we 
%follow~\cite{CDS2} and, for any 
To see this, which is crucial to our purposes, we argue as in~\cite{CDS2} and, for any given pair $(\Q,\,\M) \in \Sz \times \R^2$, 
we define the functions
\begin{equation}\label{eq:ell}
	\ell(\Q,\,\M) := \frac{1}{4}\left( \abs{\M}^2 - 1 \right)^2 
	- \beta \Q\M\cdot\M + \frac{1}{2}\left(\beta^2 
	+ \sqrt{2}\beta\right),
\end{equation}
and
\begin{equation}\label{eq:V}
	V(\M) := \ell(\Q,\,\M) - \inf_{y \in \R^2} \ell(\Q, \, y).
\end{equation}
As shown in \cite[Lemma~1.2]{CDS1}, for any given $\Q \neq 0$, 
the function $\M \mapsto V(\M)$ has exactly two zeroes, at $\M = \M_\pm$, where 
\[
	\M_\pm := \pm \left( 1 + \sqrt{2}\beta \abs{\Q} \right)^{1/2} \n,
\]
and where $\n$ is a unit vector of $\Q$ corresponding to its positive eigenvalue.

Denoting now $\Q : \Omega \to \Sz$ a $\Q$-tensor field, 
$\M : \Omega \to \R^2$ a vector field, and $G \subseteq \Omega$ 
a measurable set, we define the energy functional 
\begin{equation}\label{eq:Eeps}
	E_\eps(\M;\,G) := \int_G\left\{ \frac{\eps}{2} \abs{\nabla \M}^2 + \frac{1}{\eps} V(\M) \right\}\,{\d}x.
\end{equation}
By construction, if $(\Q_\eps,\,\M_\eps)$ is a critical point of $\F_\eps$, 
then $\M_\eps$ satisfies the Euler-Lagrange equation of $E_\eps$, i.e.,
\begin{equation}\label{eq:EL-Eeps}
	-\eps \Delta \M_\eps + \frac{1}{\eps}\nabla_\M V(\M_\eps) = 0 
	\qquad \mbox{in } \Omega.
\end{equation}
It can be easily checked that (cf. \cite[Proposition~2.3]{CDS2}) that 
\begin{equation}\label{eq:Eeps-bound}
	\sup_{\eps > 0} E_\eps(\M_\eps) \leq \mathcal{E}_0 < +\infty,
\end{equation}
where the constant $\mathcal{E}_0$ depends only on $\Omega$, $\beta$, 
and the $L^1(\partial\Omega)$- and $L^2(\partial\Omega)$-norm of 
$\Qb \times \partial_\ttau \Qb$. 

The point of introducing the auxiliary functional $E_\eps$ is that 
the Equation~\eqref{eq:EL-Eeps} coincides \emph{exactly} with~\eqref{eq:EL-M}, 
so that both $E_\eps$ and~\eqref{eq:EL-Eeps} display an Allen-Cahn 
structure. However, the results of~\cite{B} \emph{cannot} be applied directly, 
because the wells of $V$ are not fixed, but \emph{move} with both $x$ and $\eps$ 
due to their dependence on $\Q_\eps$. Nonetheless, one can still follow the strategy 
of~\cite{B}, combining it with the estimates and convergence results for 
the $\Q$-component from~\cite{CDS1} to control the perturbation terms.
%exploiting the apparatus of estimates 
%developed in~\cite{CDS1} (recalled in Section~\ref{sec:preliminary} to the 
%extent needed in this paper), in~\cite{CDS2} the arguments of~\cite{B} 
%have been adapted to obtain the following theorem.
Thus, upon defining the potential energy densities
\begin{equation}\label{eq:zeta-eps}
	\zeta_\eps := \frac{1}{\eps} V(\M_\eps), 
\end{equation}
which, due to~\eqref{eq:Eeps-bound}, are equibounded in $L^1(\Omega)$, 
one has 
\begin{equation}\label{eq:zeta*}
	\zeta_\eps \rightharpoonup^* \zeta_\star, \qquad 
	\mbox{as } \eps \to 0, 
\end{equation}
and the following theorem has been obtained in~\cite{CDS2}.
\setcounter{theorem}{-1}
\begin{theorem}[{\cite[Theorem~B]{CDS2}}]\label{thm:B-CDS}
	The set $\spt\nu_\star$ is $\H^1$-rectifiable, with locally 
	finite measure. Upon setting 
	\begin{equation}\label{eq:S*}
		\mathfrak{S}_\star := \spt\nu_\star \setminus \spt\mu_\star,
	\end{equation}
	the following holds.
	\begin{enumerate}[(i)] 
	\item\label{item:density} 
	The limiting potential energy measure $\zeta_\star$ is absolutely 
	continuous with respect to the measure $\H^1 \mres \mathfrak{S}_\star$. 
	In other words, there exists a function $\mathfrak{v}_\star : \mathfrak{S}_\star \to \R^+$, 
	locally integrable with respect to the measure $\H^1 \mres \mathfrak{S}_\star$, 
	such that
	\[
		\zeta_\star = \mathfrak{v}_\star \,\H^1 \mres \mathfrak{S}_\star.
	\] 
	The function $\mathfrak{v}_\star$ is locally bounded both from above and 
	from below in any compact set $K \subset \Omega\setminus\spt\mu_\star$.
	\item\label{item:zeta*-weight-measure} The measure $\zeta_\star$ is the weight measure of the $\H^1$-rectifiable varifold 
	$\mathbb{V}_\star$ carried by $\mathfrak{S}_\star$ with density function $\mathfrak{v}_\star$. 
	\item\label{item:balance-law} The varifold $\mathbb{V}_\star$ is stationary in $\Omega \setminus \spt\mu_\star$ and 
	its first variation as a varifold in $\Omega$ is concentrated on $\spt\mu_\star$.
	\item\label{item:segments} The set $\spt\nu_\star$ is locally a union of segments, open relative to $\Omega$, 
	each of which having constant density, given by the value of $\mathfrak{v}_\star$ at any 
	point of the segment. In addition, apart from an exceptional, $\H^1$-null set, around any point $x_0$ of 
	$\mathfrak{S}_\star$, the singular set $\spt\nu_\star$ consists of exactly one segment, 
	with constant density $\mathfrak{v}_\star(x_0)$.
	\item\label{item:unif-conv-M} As $\eps \to 0$, $\M_\eps \to \M_\star$ in $L^\infty_{\rm loc}(\Omega \setminus (\spt\mu_\star \cup \spt\nu_\star))$.
	\end{enumerate}
\end{theorem}
%
%In fact, $\spt\nu_\star$ is the support of a $\H^1$-rectifiable 
%varifold $\mathbb{V}_\star$ whose first variation in $\Omega$ is supported on the finite 
%set of points $\spt\mu_\star \cap \Omega$.
%As in~\cite{B}, the cornerstone of the proof is a pivotal monotonicity property of 
%$\zeta_\star$, i.e., that for any $x_0 \in \Omega$, the function $r \mapsto \frac{\zeta_\star(B(x_0,\,r)}{r}$ is monotone non-decreasing. As an immediate consequence, one has
%\[
%	\zeta_\star(\spt\mu_\star) = 0.
%\]
%Moreover, this monotonicity property implies item~\ref{item:density} and is one of the 
%key ingredients in the proof of item~\ref{item:balance-law}.
Notice that Theorem~\ref{thm:B-CDS} says nothing about the multiplicity of the 
limiting varifold $\mathbb{V}_\star$. 
We recall that, in the scalar Allen-Cahn theory~\cite{HutchinsonTonegawa}, the 
integrality of the limiting interface where the energy critical points concentrates 
as $\eps \to 0$ is a characteristic feature, even for the 
perturbed equation, as first shown in~\cite{Tonegawa,Tonegawa2005} for Sobolev type perturbations and then, in~\cite{NagaseTonegawa,RogerSchatzle} for $L^2$ type perturbations. 
By way of contrast, integrality is generally 
\emph{false} in the vectorial case~\cite{B, B-Cetraro}, and hence 
it requires specific analysis, depending on the particular form of the 
potential function at hand. 
Our main result in this paper, Theorem~\ref{mainthm:integrality} below,  
establishes that in the regime 
\[
	\sqrt{2}\beta > 1 
	%\qquad \mbox{or, equivalently, } \qquad \kappa_\star > 1
\] 
the varifold $\mathbb{V}_\star$ has, after rescaling by the constant $\sigma_\beta$ 
in~\eqref{eq:sigma-beta}, integer multiplicity. 
\begin{theorem}\label{mainthm:integrality}
    Let $\zeta_\star = \mathfrak{v}_\star\,\H^1 \mres \mathfrak{S}_\star$ 
    be the limiting potential energy density and let 
    $\bbV_\star = \v(\mathfrak{S}_\star,\,\mathfrak{v}_\star)$ be the 
    varifold associated with $\zeta_\star$ according to Theorem~\ref{thm:B-CDS}. 
    Let 
    \begin{equation}\label{eq:sigma-beta}
        \sigma_\beta := \frac{\sqrt{2}}{3}\left( 1 + \sqrt{2}\beta \right)^{3/2}
    \end{equation}
    and assume that $\sqrt{2}\beta > 1$. %(equivalently, that $\kappa_\star > 1$). 
    Then, the varifold 
    \[
    	\sigma_\beta^{-1} \bbV_* := \left( \mathfrak{S_*},\,\frac{\mathfrak{v}_\star}{ \sigma_\beta} \right)
    \]
    has integer multiplicity. In other words, the function $\sigma_\beta^{-1}\mathfrak{v}_\star$ takes its values into $\mathbb{N}$.
\end{theorem}
The condition $\sqrt{2}\beta > 1$ ensures such a fast, quantitative decay
with respect to $\eps$ of the 
component of $\M_\eps$ in the direction orthogonal to the eigenspace $\n_\eps$ of 
$\Q_\eps$ relative to its positive eigenvalue, that this component \emph{cannot} concentrate energy in the 
limit as $\eps \to 0$. 
This fast decay is obtained through the study of the equation satisfied by the orthogonal 
component $u_{2,\eps} := \M_\eps \cdot \n^{\perp}_\eps$ of $\M_\eps$, 
which undergoes a structural change exactly 
when $\sqrt{2}\beta > 1$, see~\eqref{eq:EL-u2} below.   
As a result, 
although the problem remains fully vectorial at the $\eps$-level, 
in the regime $\sqrt{2}\beta > 1$ it is \emph{quantitatively} close (in terms 
of $\eps$) to the  
scalar problem obtained by formally neglecting the component 
$u_{2,\eps}$ of $\M_\eps$. 
More precisely, this quantitive decay %of the orthogonal component of $\M$ 
allows us to look at the equation satisfied by   
$u_{1,\eps} := \M_\eps \cdot \n_\eps$ as a \emph{perturbed} Allen-Cahn equation of 
the type considered in~\cite{NagaseTonegawa} and, thus, to infer integrality from the 
results in~\cite{NagaseTonegawa}.

Theorem~\ref{mainthm:integrality} is, to the best of our knowledge, the first 
integrality result for the limiting varifold obtained in a vectorial Allen-Cahn problem, 
at least for a non-trivial (i.e., not decoupled already at the $\eps$-level), 
physically relevant potential energy, and for non-minimising solutions. 
In particular, it shows that integrality can occur 
if the potential energy singles out a preferred direction, as it could be 
intuitively expected. On the other 
hand, even in this seemingly easy setting, 
proving that the components along the orthogonal directions decay sufficiently 
fast turns out to be a quite technical matter. 

It would be 
interesting to further study the behaviour of the multiplicity function 
$\mathfrak{v}_\star$ in the opposite regime $\sqrt{2}\beta \leq 1$. 
Particularly 
intriguing would be a situation in which integrality does not hold in such a regime 
(or ceases to be true below a certain threshold). 
Indeed, besides the specific interest of such a deep structural change in this 
particular problem, it is conceivable that the phenomena leading to the breakdown 
and the arguments involved in proving its occurrence 
could be helpful to the long-term goal of determining  
%developments of methods that, hopefully, would lead in future 
%to single out 
sufficient conditions for integrality in the vectorial Allen-Cahn problem. 
%This would be relevant to the analysis of the important applications in which 
%vectorial Allen-Cahn functionals are involved.

\subsection*{Plan of the paper}
In Section~\ref{sec:preliminary}, we gather some results from~\cite{CDS1, CDS2}, 
as needed in this paper. In Section~\ref{sec:conv-rho-eps} we obtain a quantitative 
rate with respect 
to $\eps$ for the convergence of $\abs{\Q_\eps}$ to zero in 
$W^{1,2}_{\rm loc}(\Omega \setminus \spt\mu_\star)$
and of $\frac{\abs{\Q_\eps} - 1}{\eps}$ to the constant 
$\kappa_\star$ defined in~\eqref{eq:kappa*} in 
$L^2_{\rm loc}(\Omega \setminus \spt\mu_\star)$ (see Lemma~\ref{lemma:improved-decay-rho}). This result   
is crucial to our purposes and improves on~\cite[Lemma~3.7]{CDS1} 
(which, in turn, was a crucial ingredient in~\cite{CDS1,CDS2}). 
In Section~\ref{sec:decay-u2}, we show a quantitative decay with respect 
to $\eps$ of the $W^{1,2}$-norm of $u_{2,\eps}$ away from $\spt\mu_\star$. 
Finally, in Section~\ref{sec:integrality}, we employ the results of Section~\ref{sec:decay-u2} and Lemma~\ref{lemma:improved-decay-rho} 
to prove Theorem~\ref{mainthm:integrality}. 

%----------------------------
\section{Preliminary results}\label{sec:preliminary}

\setcounter{equation}{0}
\setcounter{theorem}{0}
\numberwithin{equation}{section}
\numberwithin{definition}{section}
\numberwithin{theorem}{section}
\numberwithin{remark}{section}
\numberwithin{example}{section}

The main result of this section is Lemma~\ref{lemma:improved-decay-rho}, 
contained in Section~\ref{sec:conv-rho-eps}. 
We will use it both in Section~\ref{sec:decay-u2} and also directly in 
the proof of Theorem~\ref{mainthm:integrality}. Before arriving to 
Lemma~\ref{lemma:improved-decay-rho}, it will be necessary to establish 
some notation and gather some auxiliary results from~\cite{CDS1, CDS2}, which 
will be needed also in the rest of the paper. In particular, 
in Section~\ref{sec:ACeps} we define the auxiliary Allen-Cahn 
energy $\mathcal{AC}_\eps$. Although this functional was already considered 
in~\cite{CMSW, CDS1}, here we introduce the new decomposition~\eqref{eq:h1-h2} 
of its potential energy $h$ in~\eqref{eq:h}, which will prove to be extremely 
useful to our purposes.

\subsection{Varifolds}
We recall some basic terminology, as needed in this paper, 
concerning the theory of varifolds. For a detailed account 
of the theory, the reader is referred to~\cite{Simon}.

\vskip5pt

A $k$-varifold~$\mathbb{V}$ 
on an open set $U \subset \R^n$ is a Radon measure 
on $G_k(U) := U \times G(n,\,k)$, where $G(n,\,k)$ 
denotes the Grassmanian manifold of $k$-dimensional 
planes in $\R^n$. The \emph{weight measure} $\norm{\mathbb{V}}$ of $\mathbb{V}$ 
is the Radon measure defined by setting
\[
	\int_U \phi(x) \,{\d} \norm{\mathbb{V}}(x) 
	:= \int_{G_k(U)} \phi(x)\,{\d}\mathbb{V}(x,\,S),
\]
for all $\phi \in C_c(U)$. 
A $k$-varifold is \emph{$\H^k$-rectifiable}
%(or \emph{$k$-rectifiable}, for short)
if there exist a %$\H^k$-measurable, 
countably $\H^k$-rectifiable set $\mathfrak{S}$ and
%an $\H^k$-measurable, 
a locally $\H^k$-integrable function
$\theta\colon\mathfrak{S} \to \R^+$ (called \emph{density function} 
or \emph{multiplicity}) such that 
\[
	\int_{G_k(U)} \varphi(x,\,S) \,{\d} \mathbb{V}(x,\,S) = 
	\int_{\mathfrak{S}} \varphi(x,\,\T_x \mathfrak{S})\,\theta(x) \,{\d}\H^k(x)
\]
for all $\varphi \in C_c(G_k(U))$, where $\T_x \mathfrak{S}$ denotes 
the approximate tangent space to $\mathfrak{S}$ at $x \in \mathfrak{S}$. 
Following the common convention, we extend $\theta$ as $0$ in $\R^n \setminus \mathfrak{S}$.  
Thus, if $\mathbb{V}$ is $\H^k$-rectifiable, then 
\[
	\norm{\mathbb{V}} = \theta \H^k \mres \mathfrak{S}.
\]
If $\theta$ takes only integer values, then $\mathbb{V}$ is called 
\emph{integral}.

Two pairs $(\mathfrak{S},\,\theta)$ and 
$(\mathfrak{S}',\,\theta')$ identify the same $\H^k$-rectifiable varifold 
if and only if the symmetric difference $\mathfrak{S} \Delta \mathfrak{S}'$ 
is a $\H^k$-null set and $\theta = \theta'$ on~$\H^k$-almost all of 
$\mathfrak{S} \cap \mathfrak{S}'$. 
These conditions define an equivalence relation, 
so that rectifiable varifolds are actually equivalence classes of 
pairs $(\mathfrak{S},\,\theta)$. 
% To emphasise the pair $(\mathfrak{S},\,\theta)$, it is customary to write
We will write $\mathbb{V} = \v(\mathfrak{S}, \, \theta)$
for the rectifiable varifold carried by~$\mathfrak{S}$
with density function~$\theta$
(see e.g.~\cite[Chapter~4]{Simon}).
The \emph{first variation} of a varifold~$\mathbb{V}$ 
is the distribution $\delta \mathbb{V}$ in $U$ satisfying
\[
	\delta\mathbb{V}(\X) := 
	\int_{\mathfrak{S}} \div_{\T_x \mathfrak{S}} \X(x) \,{\d} \H^k(x)
\]
for all $\X \in C^1_c(U,\,\R^n)$. If it happens that $\delta{\mathbb{V}}(\X) = 0$ 
for any $\X \in C^1_c(U,\,\R^n)$, then $\mathbb{V}$ is said to be \emph{stationary} 
in $U$.

%------------------------------------------------
\subsection{The zero set of the potential energy and the material constant $\kappa_\star$}

It can be checked (see \cite[Lemma~B.2]{CMSW}) that,  
for any $\eps > 0$, we have $f_\eps(\Q,\,\M) = 0$ if and only 
if $(\Q,\,\M) = (\Q^{\rm pot}_\eps,\,\M^{\rm pot}_\eps)$, where 
$\Q^{\rm pot}_\eps,\,\M^{\rm pot}_\eps$ satisfy   
\[
	\abs{\M^{\rm pot}_\eps} = \lambda_{\eps,\beta},  \qquad 
	\Q^{\rm pot}_\eps = \sqrt{2} s_{\eps,\beta}\left( \frac{\M^{\rm pot}_\eps \otimes \M^{\rm pot}_\eps}{\lambda_{\eps,\beta}^2} - \frac{\I}{2} \right), 
\]
where
\[
	s_{\eps,\beta} = 1 + \eps \kappa_\star + {\rm O}_{\eps \to 0}(\eps^2), 
	\qquad 
	\lambda_{\eps,\beta}^2 = 1 + \sqrt{2}\beta + \sqrt{2}\beta \eps \kappa_\star + {\rm O}_{\eps \to 0}(\eps^2). 
\]
Notice that
\[
	s_{\eps,\beta} \to 1, \qquad 
	\lambda_{\eps,\beta} \to \left( 1 + \sqrt{2}\beta \right)^{1/2} 
	\qquad \mbox{as } \eps \to 0. 
\]
Here, 
\begin{equation}\label{eq:kappa*}
	\kappa_\star := \frac{\beta}{2\sqrt{2}}\left( 1 + \sqrt{2}\beta \right)
\end{equation}
is a material-dependent, dimensionless constant, measuring the strength of the coupling 
between $\Q$ and $\M$ in terms of how much the interaction lifts the ground level 
of the potential at first order in $\eps$ from what it would be with the interaction  switched off.

Given boundary data as in~\eqref{hp:bc-dir} or as in~\eqref{hp:bc-mixed}, 
one easily computes that
\[
	\frac{1}{\eps^2}f_\eps(\Qb,\,\Mb) 
	= \kappa_\star^2 + \o_{\eps \to 0}(\eps^2).
\]

%-------------------------------------------------
\subsection{General properties of critical points}
By standard elliptic regularity,
any critical pair $(\Q_\eps,\,\M_\eps)$ of $\F_\eps$ is 
real-analytic in $\Omega$ and, under boundary conditions 
as in~\eqref{hp:bc-dir} or~\eqref{hp:bc-mixed}, at least 
of class $C^1$ in $\overline{\Omega}$.
Moreover, from \cite[Lemma~1.6]{CDS1} we have
\begin{align}
    & \norm{\Q_\eps}_{L^\infty(\Omega)} + \norm{\M_\eps}_{L^\infty(\Omega)} \leq C_\beta \label{eq:max-QM} \\
    & \norm{\nabla \Q_\eps}_{L^\infty(\Omega)} + \norm{\M_\eps}_{L^\infty(\Omega)} \leq \frac{C_{\beta,\Omega}}{\eps}\label{eq:max-grad-QM},
\end{align}
where the constant $C_\beta$ depends only on $\beta$ and 
the constant $C_{\beta,\Omega}$ depends only on $\beta$ and $\Omega$. 
More precisely, the proof of~\cite[Lemma~1.6]{CDS1} shows that, 
for any $\eps > 0$, the pointwise inequalities
\begin{align}
    \abs{\Q_\eps} &\leq 1 + \eps \kappa_\star + c_\beta \eps^2, \label{eq:max-princ-rho} \\
    \abs{\M_\eps}^2 &\leq 1 + \sqrt{2}\beta \abs{\Q_\eps} \label{eq:max-princ-M},
\end{align}
where $c_\beta$ depends only on $\beta$, hold globally in $\Omega$, cf.~\cite[Lemma~1.6 and Remark~1.7]{CDS1}.

\subsection{Decomposition of $\Q$-tensors}\label{sec:dec-Q}
Assume that $G \subseteq \Omega$ is an open, bounded, 
and simply connected subset of $\Omega$, with smooth 
boundary $\partial G$. Assume that $\Q : G \to \Sz$ 
is a smooth $\Q$-tensor field such that $\Q(x) \neq 0$ 
for all $x \in G$. Then, by the spectral theorem and 
the simple connectedness of $G$ and well-known lifting 
results (cf.~\cite{CDS1} and references therein), we can find a coherently 
defined and smoothly varying eigenframe of smooth unit 
eigenvectors for $\Q$, denoted $(\n,\,\m)$, so that we may write
\begin{equation}\label{eq:spectral-thm}
    \Q = \frac{\abs{\Q}}{\sqrt{2}} \left( \n \otimes \n - \m \otimes \m \right) 
    \qquad \mbox{in } G.
\end{equation}
We will often use the following notation:
\[
    \rho := \abs{\Q}.
\]
We always assume that $\n$ denotes a unit eigenvector relative to the \emph{positive} eigenvalue of $\Q$.

Also, there exists a function $\varphi : G \to \R$ such that
\begin{equation}\label{eq:neps-phieps-1}
    \n  = 
    \begin{pmatrix}
        \cos\varphi \\ 
        \sin\varphi
    \end{pmatrix}.
\end{equation}
The following relations are immediately checked by straightforward computation:
\begin{equation}\label{eq:neps-phieps-2}
    \abs{\nabla \n} = \abs{\n\times \nabla \n} = \abs{\nabla \varphi}, \qquad 
    \n \times \Delta \n = \Delta \varphi.
\end{equation}
Since $\m$ is either $\n^\perp = \begin{pmatrix} -n_2 \\ n_1 \end{pmatrix}$ or $-\n^\perp$, the same relations as in~\eqref{eq:neps-phieps-2} hold with $\n$ replaced by $\m$. Furthermore, by direct computation it is easily seen that 
\begin{equation}\label{eq:dec-grad-Q-square}
	\abs{\nabla \Q_\eps}^2 = \abs{\nabla \rho_\eps}^2 + 4 \rho^2 \abs{\nabla \varphi}^2.
\end{equation}
pointwise.

\subsection{Convergence results for $\Q_\eps$}\label{sec:Qeps}
Let $\{(\Q_\eps,\,\M_\eps)\}_\eps$ be a sequence of critical points of $\F_\eps$ 
subject to boundary conditions as in~\eqref{hp:bc-dir} or~\eqref{hp:bc-mixed} and 
satisfying~\eqref{hp:equibdd-f}. 
By \cite[Proposition~2.5 and Lemma~3.3]{CDS1}, there exists a positive number 
$\eps_*$ such that, if $B = B(x_0, \, R) \csubset \Omega \setminus \spt\mu_\star$, 
then $\abs{\Q_\eps} \geq 1/2$ in $B(x_0,\,3R/4)$. 
Consequently, in $B$ we can decompose $\Q_\eps$ 
as in~\eqref{eq:spectral-thm}, write $\n_\eps$ as in~\eqref{eq:neps-phieps-1}, 
and, for each $\eps \leq \eps_* R$, we can find a function 
$\varphi_\eps : B(x_0,\,3R/4) \to \R$ satisfying~\eqref{eq:neps-phieps-1}. 
By \cite[Step~2 of Proposition~2.5]{CDS1}, we know that
there exists a 
constant $C_p(x_0,\,R)$, depending only on $\beta$, $\Omega$, $p$, $x_0$, $R$, 
${\rm C}_{\rm pot}$, and the $L_1(\partial\Omega)$- and the $L^2(\partial\Omega)$-norm
of $\Qb \times \partial_{\ttau}\Qb$, such that 
%(cf.~\cite[inequality (2.39) in the proof of Proposition~2.5]{CDS1})
\begin{equation}\label{eq:varphi-Lp-bounds}
    \forall p \in [1,\,\infty), \qquad 
    \norm{\nabla \varphi_\eps}_{L^p(B(x_0,\,R/2))} \leq C_{p}(x_0,\,R).
\end{equation}
By a standard covering argument, this implies that 
for every compact set $K \subset \Omega \setminus \spt\mu_\star$ 
there exists $\eps_K >0$ such that, for any $\eps \leq \eps_K$, 
there holds $\abs{\Q_\eps} \geq 1/2$ in $K$. In particular, 
both the considerations in Section~\ref{sec:dec-Q} and those above are valid in 
any simply connected, open subset $G \csubset K$ with smooth boundary, 
for any $\eps \leq \eps_K$. 
%Consequently, we can decompose $\Q_\eps$ 
%as in~\eqref{eq:spectral-thm}, write $\n_\eps$ as in~\eqref{eq:neps-phieps-1}, 
%and, for each $\eps \leq \eps_K$, we can find a function 
%$\varphi_\eps : G \to \R$ satisfying~\eqref{eq:neps-phieps-1}. 
%By \cite[Step~2 of Proposition~2.5]{CDS1}, we know that 
%\begin{equation}\label{eq:varphi-Lp-bounds-K}
%    \forall p \in [1,\,\infty), \qquad 
%    \norm{\nabla \varphi_\eps}_{L^p(K)} \leq C_{p}(K),
%\end{equation}
%where the constant $C_{p}(K)$ does not depend on $\eps$, 
%but only on $\beta$, $\Omega$, $p$, $K$, 
%${\rm C}_{\rm pot}$, and the $L_1(\partial\Omega)$- and the $L^2(\partial\Omega)$-norm
%of $\Qb \times \partial_{\ttau}\Qb$. 

Regarding $\rho_\eps$, we observe that it satisfies the 
equation
\begin{equation}\label{eq:rho-eps}
	-\frac{1}{2} \Delta \rho_\eps^2 + \abs{\nabla \Q_\eps}^2 + 
	\frac{1}{\eps^2}\left(\rho_\eps^2 - 1 \right) \rho_\eps^2 
	- \frac{\beta}{\eps} \Q \M \cdot \M = 0,   
\end{equation}
which is obtained by scalar multiplication of~\eqref{eq:EL-Q} 
with $\Q_\eps$. If $G \subset \Omega$ is an open set such that 
$\rho_\eps > 0$ in $G$, then, upon using~\eqref{eq:dec-grad-Q-square} 
and dividing both sides by $\rho_\eps$, 
we can recast~\eqref{eq:rho-eps} as
\begin{equation}\label{eq:rho-eps-bis}
	-\Delta \rho_\eps + 4 \rho^2 \abs{\nabla \varphi}^2 + 
	\frac{1}{\eps^2}\left(\rho_\eps^2 - 1 \right) \rho_\eps^2 
	- \frac{\beta}{\eps} \Q \M \cdot \M = 0 \qquad \mbox{in } G.
\end{equation} 
By using~\eqref{eq:varphi-Lp-bounds} and testing suitably~\eqref{eq:rho-eps-bis}, 
it has been shown in \cite[Proposition~2.6]{CDS1} that 
\begin{equation}\label{eq:GL-bounds-p}
	\norm{\nabla \Q_\eps}_{L^p(B(x_0,\,R/2))} + \norm{\frac{\rho_\eps - 1}{\eps}}_{L^p(B(x_0,\,R))} \leq C_{p}(x_0,\,R),
\end{equation}
for a constant $C_p(x_0,\,R)$ depending only on 
$\beta$, $\Omega$, $p$, $x_0$, $R$, 
${\rm C}_{\rm pot}$, and the $L_1(\partial\Omega)$- and the $L^2(\partial\Omega)$-norm
of $\Qb \times \partial_{\ttau}\Qb$.
%Moreover, by covering and recalling \cite[Remark~2.7]{CDS1}, for every 
%compact set $K \subset \Omega \setminus \spt\mu_\star$ and every 
%$p \in [1,+\infty)$ there exists a positive constant $C_{p}(K)$, 
%depending only on $\beta$, $\Omega$, $p$, $K$, 
%${\rm C}_{\rm pot}$, and the $L_1(\partial\Omega)$- and the $L^2(\partial\Omega)$-norm
%of $\Qb \times \partial_{\ttau}\Qb$, such that
%\begin{equation}\label{eq:GL-bounds-p}
%	\norm{\nabla \Q_\eps}_{L^p(K)} + \norm{\frac{\rho_\eps - 1}{\eps}}_{L^p(K)} 
%	\leq C_{p}(K).
%\end{equation}
Furthermore, it has been shown in \cite[Lemma~2.8]{CDS1} that 
\begin{equation}\label{eq:rho-unif-conv}
    \rho_\eps \to 1 \qquad \mbox{uniformly in } B.
\end{equation}
for every ball $B \csubset \Omega \setminus \spt\mu_\star$ 
and hence, by covering, in every compact set $K \subset \Omega\setminus \spt\mu_\star$. 
In addition, by \cite[Lemma~3.7]{CDS1}, we have  
\begin{equation}\label{eq:strong-conv-rho}
    \int_B \left\{\abs{\nabla \rho_\eps}^2 + \left(\frac{\rho_\eps - 1}{\eps} - \kappa_\star\right)^2\right\}\,{\d}x \to 0 \qquad \mbox{as } \eps \to 0
\end{equation}
for any ball $B \csubset \Omega \setminus \spt \mu_\star$. 
Along with the uniform bounds in \cite[Proposition~2.6]{CDS1},~\eqref{eq:strong-conv-rho} 
implies that for every $p \in [1, +\infty)$ we have (cf.~\cite[Remark~3.6]{CDS1})
\begin{equation}\label{eq:strong-conv-rho-p}
    \int_B \left\{\abs{\nabla \rho_\eps}^p + \abs{\frac{\rho_\eps - 1}{\eps} - \kappa_\star}^p\right\}\,{\d}x \to 0 \qquad \mbox{as } \eps \to 0, 
\end{equation}
for every $p \in [1,\,+\infty)$ and every ball $B \csubset \Omega \setminus \spt \mu_\star$. 

Among the most important consequences of~\eqref{eq:strong-conv-rho-p}, there is 
the fact, crucially exploited in~\cite{CDS2}, that the limiting potential 
energy density $\zeta_\star$ defined in~\eqref{eq:zeta*}   
contains, at first order, all the asymptotic information about the full potential 
energy density $\frac{1}{\eps^2} f_\eps(\Q_\eps,\,\M_\eps)$, in the sense 
that, for any $p \in [1,\,+\infty)$, there holds
\begin{equation}\label{eq:f-V-p}
	\int_K \abs{\frac{1}{\eps^2}f_\eps(\Q_\eps,\,\M_\eps) - \frac{1}{\eps}V(\M_\eps)}^p \,{\d}x \to 0, \qquad \mbox{as } \eps \to 0.
\end{equation}
Lemma~\ref{lemma:improved-decay-rho} below improves on~\eqref{eq:strong-conv-rho},   
so as to promote the above convergence to 
a \emph{quantitative decay} with respect to $\eps$.

%---------------------------------------
\subsection{Decomposition of $\M$ along an eigenframe of $\Q$: the map $\u$}

Let $G \subseteq \Omega$ be a simply connected, open set with smooth boundary $\partial G$. Assume that $\Q : G \to \Sz$ is a $\Q$-tensor field such that $\Q \neq 0$ in $G$. Let $(\n,\,\m)$ denote any eigenframe associated with $\Q$ as in~\eqref{eq:spectral-thm}. 
Then, we can project $\M$ along the eigenframe $(\n,\,\m)$, obtaining the map $\u : G \to \R^2$ defined by
\begin{equation}\label{eq:u}
   \u := \left(\M\cdot \n,\,\M\cdot \m\right). 
\end{equation}
We set
\begin{equation}\label{eq:u1u2}
    u_1 := \M \cdot \n, \quad 
    u_2 := \M \cdot \m
\end{equation}
and notice that
\begin{equation}\label{eq:mod-u}
    \abs{\u} = \abs{\M}
\end{equation}
pointwise in $G$. In particular, by~\eqref{eq:mod-u} and~\eqref{eq:max-QM}, 
\begin{equation}\label{eq:norm-u}
    \norm{\u}_{L^\infty(G)} \leq \norm{\M}_{L^\infty(\Omega)} \leq C_\beta.
\end{equation}
We also observe that, if $\Q$, $\M \in C^\infty(G)$, then $\u \in C^\infty(G)$. 
For later purposes, we explicitly observe that if $\Q \in \Sz \setminus \{0\}$, 
then
\[
	\Q \M \cdot \M = \frac{\rho}{\sqrt{2}}\left( (\M \cdot \n)^2 - (\M \cdot \m)^2 \right), 
\]
i.e.,
\begin{equation}\label{eq:QMM-u}
	\Q \M \cdot \M = \frac{\rho}{\sqrt{2}}\left( u_1^2 - u_2^2 \right)
\end{equation}

%---------------------------------------------------------------
\subsection{The auxiliary energy functional $\mathcal{AC}_\eps$}
\label{sec:ACeps}
Let $\{(\Q_\eps,\,\M_\eps)\}_\eps$ be a sequence of critical points 
of $\F_\eps$ satisfying~\eqref{hp:equibdd-f}. 
Let $K \subset \Omega \setminus \spt \mu_\star$ be any compact set, 
and let $G \subset K$ be any simply connected, open set, with smooth boundary. 
As already seen above, there exists $\eps_K$ so that, 
for any $\eps \leq \eps_K$, there holds $\abs{\Q_\eps} \geq 1/2$ in $K$. 
Thus, for each $\eps \leq \eps_K$, from $\M_\eps$ and $\n_\eps$, we can construct a map $\u_\eps : G \to \R^2$ as in~\eqref{eq:u}. Upon defining 
\begin{equation}\label{eq:ACeps}
    \mathcal{AC}_\eps(\u_\eps;\,G) := \int_G \left\{ \frac{\eps}{2}\abs{\nabla \u_\eps}^2 + \frac{1}{\eps} h(\u_\eps) \right\}\,{\d}x,
\end{equation}
where
\begin{equation}\label{eq:h}
    h(\u) = h(u_1,\,u_2)  := \frac{1}{4}\left( \abs{\u}^2 - 1 \right)^2 - \frac{\beta}{\sqrt{2}}\left( u_1^2 - u_2^2 \right) + \frac{\beta^2 + \sqrt{2}\beta}{2},
\end{equation}
we can rewrite (cf., e.g.,~\cite[Proposition~3.2]{CMSW})
\[
    \F_\eps(\Q_\eps,\,\M_\eps;\,G) = 
    \int_G \left\{ \frac{1}{2} \abs{\nabla \Q_\eps}^2 + \frac{1}{4\eps^2}\left( \abs{\Q_\eps}^2 - 1  \right)^2\right\}\,{\d}x 
    + \mathcal{AC_\eps}(\u_\eps;\,G) 
    + R_\eps,
\]
where
\[
	R_\eps \to 0 \qquad \mbox{as } \eps \to 0.
\]
By~\cite[Lemma~3.2]{CDS1}, we know that
\begin{equation}\label{eq:ACeps-equibdd}
    \sup_{\eps > 0} \mathcal{AC}_\eps(\u_\eps;\,G) \leq C < +\infty,
\end{equation}
where the constant $C$ does not depend on $\eps$ and on $G$, but 
only on $\beta$, $\Omega$, $\Cpot$, and the $L^1(\partial\Omega)$- 
and the $L^2(\partial\Omega)$-norm of $\Qb \times \partial_\ttau \Qb$. 
Since, as it can be easily checked,  
\[
    h(\u) = h(\abs{\u},\,0) + \sqrt{2}\beta u_2^2, \qquad h(\abs{\u},\,0) \geq 0, 
\]
for any vector $\u = (u_1,\,u_2) \in \R^2$, 
an immediate consequence of~\eqref{eq:ACeps-equibdd} is the bound
\begin{equation}\label{eq:decay-u2-basic}
    \frac{1}{\eps}\int_G u_{2,\eps}^2 \,{\d}x \leq C_2,
\end{equation}
where the constant $C_2$ does not depend on $\eps$ and on $G$, but 
only on $\beta$, $\Omega$, $\Cpot$, and the $L^1(\partial\Omega)$- 
and the $L^2(\partial\Omega)$-norm of $\Qb \times \partial_\ttau \Qb$. 

We point out some additional remarks on the function $h$. 
By straightforward computations, we see that $h(\u)$ can 
be rewritten in terms of the components $u_1$, $u_2$ in the following 
ways (both of them will be used later on): 
\begin{equation}\label{eq:h-ter}
	h(\u) = \frac{1}{4}\left( \abs{\u}^2 - 1 - \sqrt{2}\beta \right)^2 + \sqrt{2}\beta u_2^2
\end{equation}
and also
\begin{equation}\label{eq:h-bis}
    h(u_1,\,u_2) = \frac{1}{4}\left(  u_1^2 - 1 - \sqrt{2} \beta \right)^2 + u_2^2 \left( \frac{1}{2}\left( u_1^2 - 1 + \sqrt{2} \beta \ \right) + \frac{1}{4}u_2^2 \right).
\end{equation}
For $u$, $v \in \R$, we define
\begin{equation}\label{eq:h1-h2}
    h_1(u) := \frac{1}{4}\left( u^2 - 1- \sqrt{2}\beta \right)^2, \qquad 
    h_2(u,\,v) := v^2\left( \frac{1}{2}\left(u^2 - 1 + \sqrt{2}\beta \right) + \frac{1}{4}v^2 \right),
\end{equation}
so that for any $\u = (u_1,\,u_2) \in \R^2$ we have
\[
    h(\u) = h_1(u_1) + h_2(u_1,\,u_2).
\]
If $\sqrt{2}\beta \geq 1$, then the function $h_2$ is non-negative, 
for every $u$, $v \in \R$.
Thus, if $\sqrt{2}\beta \geq 1$, then $h$ is the sum of two positive terms, one of which involving only $u_1$.
Moreover, besides~\eqref{eq:decay-u2-basic}, we have other 
immediate consequences of~\eqref{eq:ACeps-equibdd}, that is, 
\begin{gather}
	\frac{1}{\eps} \int_G \left(  u_1^2 - 1 - \sqrt{2} \beta \right)^2 \,{\d}x \leq C_1, \label{eq:decay-u1-to-well} \\
	\frac{1}{\eps} \int_G u_2^4 \,{\d}x \leq C_2. \label{eq:decay-u2^4-basic}
\end{gather}
For later use, we also observe that, by the definition~\eqref{eq:V} of $V(\M_\eps)$, the 
definition~\eqref{eq:h} of $h(\u)$, and elementary computations, 
it follows that (cf. \cite[Lemma~2.1]{CDS2})
\begin{equation}\label{eq:V-h}
	V(\M) - h(\u) 
	= \frac{\beta}{\sqrt{2}}\left(1 - \abs{\Q} \right)\left(u_1^2 - u_2^2 - \frac{\beta + \beta\abs{\Q}}{\sqrt{2}} \right).
\end{equation}
If $(\Q_\eps,\,\M_\eps)$ is a critical pair for $\F_\eps$, 
then the equation satisfied by $\u_\eps$ in $G$ reads as follows:
\begin{equation}\label{eq:EL-u}
    -\Delta \u_\eps + \frac{1}{\eps^2}\left(\abs{\u_\eps}^2-1\right) \u_\eps - \frac{\sqrt{2}\beta \abs{\Q_\eps}}{\eps^2} \bar{\u}_\eps = \u_\eps \abs{\nabla \varphi_\eps}^2 - \nabla \u_\eps^\perp \cdot \nabla \varphi_\eps - \u_\eps^\perp \Delta \varphi_\eps \qquad \mbox{in } G,
\end{equation}
where we used the notation
\[
   \u^\perp = 
   \begin{pmatrix}
    -u_2 \\
     u_1
   \end{pmatrix}, \qquad 
   \bar{\u} = 
      \begin{pmatrix}
     u_1 \\
     -u_2
   \end{pmatrix}.
\]
Equation~\eqref{eq:EL-u} is obtained by rewriting~\eqref{eq:EL-M} in terms of $\u$. 
Note, in particular, that~\eqref{eq:EL-u} is \emph{not} the Euler-Lagrange 
equation associated with the functional $\mathcal{AC}_\eps$. 
In fact, it differs from it because the right-hand side of~\eqref{eq:EL-u} 
is not zero (rather, it contains the interaction terms with the $\Q_\eps$-components).

%----------------------------------------------------------------
\subsection{Quantitative convergence for the modulus of ${\Q_\eps}$}
\label{sec:conv-rho-eps}

For our purposes in this paper, we need an improved version of~\eqref{eq:strong-conv-rho}, 
which is contained in Lemma~\ref{lemma:improved-decay-rho} below. 
\begin{lemma}\label{lemma:improved-decay-rho}
    For every compact set $K \subset \Omega \setminus \spt\mu_\star$, 
    there holds
    \begin{equation}\label{eq:improved-decay-rho}
        \int_K \left\{ \abs{\nabla \rho_\eps}^2 + \left(\frac{\rho_\eps -1}{\eps} - \kappa_\star \right)^2 \right\}\,{\d}x \leq C_{\beta,K} \eps \qquad \mbox{as } \eps \to 0.
    \end{equation}
The constant $C_{\beta,K}$ depends only on $\beta$, $K$, $\Omega$, $\Cpot$, and 
the $L^1(\partial\Omega)$- and the $L^2(\partial\Omega)$-norm of 
$\Qb \times \partial_{\ttau} \Qb$.
\end{lemma}

\begin{proof}	
	Using a standard covering argument, it suffices to show 
	that, for any ball $B = B(x_0,\,R) \csubset \Omega \setminus \spt\mu_\star$, we have 
	\begin{equation}\label{eq:improved-decay-rho-B}
        \int_{B'} \left\{ \abs{\nabla \rho_\eps}^2 + \left(\frac{\rho_\eps -1}{\eps} - \kappa_\star \right)^2 \right\}\,{\d}x \leq C_\beta(x_0,\,R) \eps, 
    \end{equation}
	where $B' = B(x_0,\,R/2)$ and $C_\beta(x_0,\,R)$ 
	depends only on $\beta$, $x_0$, $R$, $\Omega$, $\Cpot$, and 
	the $L^1(\partial\Omega)$- and the $L^2(\partial\Omega)$-norm of 
	$\Qb \times \partial_{\ttau} \Qb$.
	
	\setcounter{step}{0}    
	\begin{step}[Using the clearing-out property]
	Since $B$ stays at positive distance from $\spt\mu_\star$, 
	we know from \cite[Proposition~2.5 and Lemma~3.3]{CDS1} that there exists $\eps_*$ 
    such that, for any $\eps$ with $0 < \eps \leq \eps_* R$, 
    we have $\rho_\eps \geq 1/2$ in $B$. 
    From now on, to keep the notation as light as possible, 
    we drop the subscript $\eps$, writing (for instance)
    $\rho$ instead of $\rho_\eps$, and so on.
    
	Since $\rho \geq 1/2$ in $B$, by~\eqref{eq:rho-eps-bis} 
	and~\eqref{eq:QMM-u}, the equation satisfied by $\rho$ in $B$ reads
    \begin{equation}\label{eq:eq-rho}
        -\Delta \rho + 4\rho \abs{\nabla \varphi}^2 
        + \frac{1}{\eps^2}(\rho - 1) (\rho + 1)\rho 
        = \frac{\sigma}{\eps},
    \end{equation}
    where we set
    \[
        \sigma := \beta \frac{\Q \M\cdot \M}{\rho} 
        = \frac{\beta}{\sqrt{2}}\left(u_1^2 - u_2^2\right).
    \]
    As in \cite[proof of Lemma~3.7]{CDS1}, we subtract $\frac{2}{\eps} \kappa_\star$ from both 
    sides of~\eqref{eq:eq-rho}, so as to obtain
    \begin{equation}\label{eq:eq-rho-bis}
        -\Delta \rho + 4 \rho \abs{\nabla \varphi}^2 + \frac{1}{\eps}\left(\frac{\rho - 1}{\eps} (\rho + 1)\rho - 2 \kappa_\star \right) = \frac{\sigma - 2 \kappa_\star}{\eps}.
    \end{equation}
    We further observe that
    \begin{equation}\label{eq:sigma-2k*}
    \begin{split}
    	\sigma - 2 \kappa_\star &= \frac{\beta}{\sqrt{2}}\left( u_1^2 - u_2^2 - 1 - \sqrt{2}\beta \right) \\
    	&=\frac{\beta}{\sqrt{2}}\left( \abs{\u}^2 - 1 - \sqrt{2}\beta - 2 u_2^2 \right), 
 %   	&=\frac{\beta}{\sqrt{2}}\left( \abs{\u}^2 - 1 - \sqrt{2}\beta\rho + \sqrt{2}\beta(1-\rho) - 2u_2^2  \right) \\
 %   	&=\frac{\beta}{\sqrt{2}}\left( \abs{\M}^2 - 1 - \sqrt{2}\beta\rho + \sqrt{2}\beta(1-\rho) - 2u_2^2  \right) \\
 %   	&\leq \frac{\beta}{\sqrt{2}}\left( \sqrt{2}\beta(1 + \eps \kappa_\star -\rho) - \sqrt{2}\beta \eps \kappa_\star \right)\\
 %  	&\leq \beta^2 \left( \rho - 1 - \eps \kappa_\star \right)_-,
    \end{split}
    \end{equation}
 %   where we used~\eqref{eq:max-princ-M} to pass from the third to the fourth line above.
 	which will be used later on. We will also use the trivial identity 
    \begin{equation}\label{eq:trivial-rho}
    	(\rho - 1)(\rho + 1) \rho = 2(\rho - 1) + 3(\rho - 1)^2 + (\rho-1)^3. 
    \end{equation}
    \end{step}
    
\begin{step}
	Let $\zeta \in C^\infty_c(\R^2)$ be any cut-off function with support in $B$ 
    such that 
    \[
    	\spt \zeta \subset B, \qquad 
    	0 \leq \zeta \leq 1, \qquad 
    	\zeta \equiv 1 \quad \mbox{in } B', \qquad 
    	\abs{\nabla \zeta} \leq \frac{4}{R} \quad \mbox{in } \R^2.
    \] 
	We multiply both sides of~\eqref{eq:eq-rho-bis} by 
	$\zeta^2 (\rho - 1 - \eps \kappa_\star)$ and we integrate 
	the result over $B$. After rearrangement and upon 
	using~\eqref{eq:trivial-rho}, we obtain
	\begin{equation}\label{eq:improved-decay-rho-compu1}
	\begin{split}
		\int_B \zeta^2 \left\{ \abs{\nabla \rho}^2 + 2 \left(\frac{\rho-1}{\eps} - \kappa_\star\right)^2 \right\} \,{\d}x 
		&= \int_B \zeta^2 (\sigma - 2 \kappa_\star) \left(\frac{\rho-1}{\eps} - \kappa_\star\right)\,{\d}x \\
		&- \eps \int_B 2 \zeta \, \nabla \rho \cdot \nabla \zeta  \left(\frac{\rho-1}{\eps} - \kappa_\star\right) \,{\d}x \\
		&- \eps \int_B 4 \zeta^2 \, \rho \abs{\nabla \varphi}^2 \left(\frac{\rho-1}{\eps} - \kappa_\star\right)  \,{\d}x \\
		&- \int_B \zeta^2\left( 2 (\rho-1)^2 + (\rho-1)^3 \right)\left(\frac{\rho-1}{\eps} - \kappa_\star\right) \,{\d}x
	\end{split}
	\end{equation}
	Upon using Young's inequality, we can estimate
	\begin{equation}\label{eq:improved-decay-rho-compu2}
		\eps \int_B 2 \zeta \nabla \rho \cdot \nabla \zeta  \left(\frac{\rho-1}{\eps} - \kappa_\star\right) \,{\d}x 
		\leq \frac{1}{2} \int_B \zeta^2 \abs{\nabla \rho}^2\,{\d}x + 8 \eps^2 \int_B \abs{\nabla \zeta}^2 \left(\frac{\rho-1}{\eps} - \kappa_\star\right)^2 \,{\d}x. 
	\end{equation}
	Since $\abs{\nabla \zeta}^2 \leq 16/R^2$, we have  
	\[
		\int_B \abs{\nabla \zeta}^2 \left(\frac{\rho-1}{\eps} - \kappa_\star\right)^2 \,{\d}x 
		\leq \frac{16}{R^2} \int_B \left(\frac{\rho-1}{\eps} - \kappa_\star\right)^2 \,{\d}x.
	\]
	By~\eqref{eq:GL-bounds-p} (with $p = 2$), 
	the integral on the right-hand side 
	is bounded independently of $\eps$, in terms of a constant $C_\beta(x_0,\,R)$ 
	depending only on $\beta$, $\Omega$, $x_0$, $R$, ${\rm C}_{\rm pot}$, and the 
	$L_1(\partial\Omega)$- and the $L^2(\partial\Omega)$-norm
of $\Qb \times \partial_{\ttau}\Qb$. Therefore,
	\[
	\eps \int_B 2 \zeta \nabla \rho \cdot \nabla \zeta  \left(\frac{\rho-1}{\eps} - \kappa_\star\right) \,{\d}x 
		\leq \frac{1}{2} \int_B \zeta^2 \abs{\nabla \rho}^2\,{\d}x + C_\beta(x_0,\,R) \eps^2.
	\]
%	if we shrink $\bar{\eps}$ so 
%	that $\bar{\eps} \leq R / 4$, say, then we make the  
%	right-hand side of~\eqref{eq:improved-decay-rho-compu2}  
%	smaller than half of the left-hand side of~\eqref{eq:improved-decay-rho-compu1}.
	Consequently, using again Young's inequality, 
	from~\eqref{eq:improved-decay-rho-compu1} 
	we obtain , 
	\begin{equation}\label{eq:improved-decay-rho-compu3}
	\begin{split}
		\int_B \zeta^2 \left\{ \abs{\nabla \rho}^2 + \left(\frac{\rho-1}{\eps} - \kappa_\star\right)^2 \right\} \,{\d}x 
		\leq & 3 \int_B \zeta^2 (\sigma - 2 \kappa_\star)^2 \,{\d}x 
		+ 12 \eps^2 \int_B \zeta^2 \rho^2 \abs{\nabla \varphi}^4   \,{\d}x \\
		&+ 3 \int_B \zeta^2 \left( 2 (\rho-1)^2 + (\rho-1)^3 \right)^2 \,{\d}x 
		+ C_\beta(x_0,\,R) \eps^2
	\end{split}
	\end{equation}
	From~\eqref{eq:GL-bounds-p} (i.e., from \cite[Proposition~2.6]{CDS1}), we already 
	know that the last term on the right-hand side 
	in~\eqref{eq:improved-decay-rho-compu1} tends to zero at least 
	as fast as $\eps^4$, i.e., 
	\begin{equation}\label{eq:improved-decay-rho-compu4}
		\int_B \zeta^2 \left( 2 (\rho-1)^2 + (\rho-1)^3 \right)^2 \,{\d}x
		= \mathrm{O}_{\eps \to 0}(\eps^4).
	\end{equation}	   
	Next, by~\eqref{eq:varphi-Lp-bounds} and~\eqref{eq:max-princ-rho}, we have 
	\begin{equation}\label{eq:improved-decay-rho-compu5}
		\eps^2 \int_B 4 \zeta^2 \rho^2 \abs{\nabla \varphi}^4 \,{\d}x 
		= \mathrm{O}_{\eps \to 0}(\eps^2).
	\end{equation}
	Finally, by~\eqref{eq:sigma-2k*} and~\eqref{eq:h-ter}, 
	\[
		\frac{1}{4} (\sigma - 2 \kappa_\star)^2
		\leq c_\beta\left[ \frac{1}{4}\left(\abs{\u}^2 - 1 - \sqrt{2}\beta \right)^2 + \sqrt{2}\beta u_2^2 \right]
		= c_\beta h(\u),
	\]
	where $c_\beta$ is a constant depending only on $\beta$. 
	Therefore, by~\eqref{eq:ACeps-equibdd} and since $0 \leq \zeta \leq 1$,
	\begin{equation}\label{eq:improved-decay-rho-compu6}
		\frac{1}{4} \int_B \zeta^2 (\sigma - 2 \kappa_\star)^2 \,{\d}x = \mathrm{O}_{\eps \to 0} (\eps).
	\end{equation}
	Combining~\eqref{eq:improved-decay-rho-compu3},~\eqref{eq:improved-decay-rho-compu4}, 
	\eqref{eq:improved-decay-rho-compu5}, and~\eqref{eq:improved-decay-rho-compu6} 
	with~\eqref{eq:improved-decay-rho-compu2} and 
	recalling that $\zeta^2 \equiv 1$ on $B'$, 
	we end up with~\eqref{eq:improved-decay-rho-B}.
\end{step}

\end{proof}

%{\BBB To further improve this lemma, we should 
%estimate $\int_B \zeta^2(\sigma - 2 \kappa_\star)\left( \frac{\rho-1}{\eps} - \kappa_\star \right)$
%without using, as I did, Young's or H\"{o}lder's inequality. 
%This is because the term  
%$\int_B \zeta^2 \left(\abs{\u}^2 - 1 - \sqrt{2}\beta \right)^2$ 
%is at most of order $\eps$, but it cannot be of lower order than that, in 
%general. However, it could be possible to prove~\eqref{eq:improved-decay-rho} 
%for any $p < +\infty$. This would probably be helpful, because it would 
%imply, via Gagliardo-Nirenberg inequality, the (local) uniform convergence 
%of $\frac{\rho - 1}{\eps}$ to $\kappa_\star$, which, at the moment being, 
%is not clear. 
%}

%--------------------------------------------
\section{Refined asymptotics in the regime $\sqrt{2}\beta > 1$}\label{sec:decay-u2}

In this section, we show that, if $\sqrt{2}\beta > 1$,  
%(or, alternatively, 
%for any value of $\beta > 0$, but locally away from $\spt\mu_\star \cup \spt\nu_\star$ 
%--- see Remark~\ref{rk:decay-away-from-singular-sets}), 
then the structure of the equation satisfied by the component $u_{2,\eps}$ implies 
the decay of the $W^{1,2}$-norm of $u_{2,\eps}$ at a power-rate in $\eps$, 
far from $\spt\mu_\star$. 
As a consequence, in this regime the component $u_{2,\eps}$ does \emph{not} 
concentrate energy on the limiting varifold.
%the problem associated with the functional $\mathcal{AC}_\eps$ and the equation~\eqref{eq:EL-u} looks pretty much like a \emph{perturbed} but \emph{scalar} Allen-Cahn problem.
The main result of this Section is the following proposition.
\begin{prop}\label{prop:decay-u2}
	Let $\{(\Q_\eps,\,\M_\eps)\}_\eps$ be a sequence of critical points of 
	$\F_\eps$ satisfying boundary conditions either as in \eqref{hp:bc-dir} or as 
	in~\eqref{hp:bc-mixed}, as well as assumption~\eqref{hp:equibdd-f}. 
	Assume, in addition, that $\sqrt{2}\beta > 1$.
    Let $K \subset \Omega \setminus \spt\mu_\star$ be any compact set   
     % for some $c > 0$. 
    %or, alternatively, assume that $K \subset \Omega \setminus (\spt\mu_\star \cup \spt\nu_\star)$.
    and let $G \csubset K$ be any simply connected, open set, with smooth 
    boundary $\partial G$.   
    Let $\{\u_\eps\}_\eps$ be the sequence of maps defined in $G$ as prescribed by~\eqref{eq:u}. Then, for any $\alpha \in (0,\,2)$, there hold
    \begin{gather}
        \int_G \left( \eps \abs{\nabla u_{2,\eps}}^2 + \frac{1}{\eps}\left(\abs{\u_\eps}^2 - 1 + \sqrt{2}\beta \rho_\eps \right) u_{2,\eps}^2 \right)\,{\d}x \leq C_\alpha(\dist(G, \partial K)) \eps^\alpha, \label{eq:decay-u2-goal1} \\ 
        \int_G \left( \eps \abs{\nabla u_{2,\eps}}^2 + \frac{1}{\eps} u_{2,\eps}^2 \right)\,{\d}x \leq C_\alpha'(\dist(G, \partial K)) \eps^\alpha \label{eq:decay-u2-goal2}.
    \end{gather}
    The constants $C_\alpha(\dist(G,\,\partial K))$, $C'_\alpha(\dist(G,\,\partial K))$ depend only on $\dist(G, \partial K)$, $\alpha$, $\beta$, $\Omega$, $\Cpot$, and the $L^1(\partial\Omega)$- and the $L^2(\partial\Omega)$-norms of $\Qb \times \partial_\ttau \Qb$. 
\end{prop}

\begin{proof}
    For notational convenience, we shall drop all subscripts $\eps$ within this proof. 

    We first argue locally, proving a version 
    of~\eqref{eq:decay-u2-goal1} and~\eqref{eq:decay-u2-goal2} 
    that hold in balls whose closure is well contained in 
    $\Omega \setminus \spt \mu_\star$. 
    Then, we obtain the inequalities in the desired form by a 
    standard covering argument. 
    
    Let $B_1 := B(x_0,\,R)$ be any ball in $\Omega \setminus \spt\mu_\star$ 
    such that $B_2 := B(x_0,\,2R)$ is still contained in $\Omega \setminus \spt\mu_\star$. 
    For $\ell > 0$, let us denote $B_{\ell} := B(x_0,\,\ell R)$. 
    By \cite[Corollary~3.8]{CDS1}, we know that there exists $\eps_1 = \eps_1(x_0,\,R)$ 
    such that $\abs{\Q_\eps} \geq 1/2$ in $B_1$, for every $\eps \leq \eps_1$. 
    Consequently, defining pointwise $\u$ as in~\eqref{eq:u} 
    gives rise to a well-defined map in $B_1$ which satisfies~\eqref{eq:EL-u}. 
    The equation satisfied by the component $u_2$ is obtained by projecting~\eqref{eq:EL-u} along the direction $\e_2 := (0,\,1)^{\rm t}$ and, after adding to both sides the term $\frac{\sqrt{2}\beta(1+\eps \kappa_\star)}{\eps^2} u_2$ 
    and rearranging, it reads as follows:
    \begin{equation}\label{eq:EL-u2}
    \begin{split}
        -\Delta u_2 &+ \frac{1}{\eps^2} \left( \abs{\u}^2 - 1 + \sqrt{2}\beta + \sqrt{2}\beta \eps \kappa_\star \right) u_2 \\
        &= u_2 \abs{\nabla \varphi}^2 + \nabla u_1 \cdot \nabla \varphi + u_1 \Delta \varphi + \sqrt{2}\beta  \left(\kappa_\star - \frac{\rho_\eps - 1}{\eps}\right)\frac{u_2}{\eps} \qquad \mbox{in } B_1.
    \end{split}
    \end{equation}
    By assumption, there exists some $c > 0$ such that $\sqrt{2}\beta = 1 + c$, 
    hence 
    \[
    	\left( \abs{\u}^2 - 1 + \sqrt{2}\beta + \sqrt{2}\beta \eps \kappa_\star \right) \geq c > 0
    \] 
    for every $\eps > 0$.
 %   Let $K \subset \Omega \setminus \spt \mu_\star$ be any compact set and 

 \setcounter{step}{0}
\begin{step}[Basic bounds]\label{step:u2-basic-bounds}
    Let us denote
    \[
        B = B_1, \qquad B' = B_{1/2}, \qquad B'' = B_{1/4}
    \]
    Let $\zeta \in C^\infty_c(\R^2)$ be any cut-off function such that
    \begin{gather}
        0 \leq \zeta \leq 1, \qquad 
        \zeta \equiv 1 \quad \mbox{in } B', \qquad 
        \spt\zeta \subset B, \label{eq:cut-off} \\ 
        \abs{\nabla \zeta} \leq 2\left(\frac{2}{R}\right), \qquad \abs{\nabla^2 \zeta} \leq 2\left(\frac{4}{R^2}\right) \qquad \mbox{in } \R^2. \label{eq:cut-off-grad}
    \end{gather}
    After multiplying~\eqref{eq:EL-u2} by $\eps \zeta^2 u_2$ and 
    integrating the result over $\Omega$ (or, which is the same, over $B$), 
    we obtain the identity
    \begin{equation}\label{eq:decay-u2-compu1}
    \begin{split}
        \int_\Omega & \zeta^2 \left( \eps \abs{\nabla u_2}^2 + \frac{1}{\eps} \left( \abs{\u}^2 - 1 + \sqrt{2}\beta + \sqrt{2}\beta \eps \kappa_\star \right) u_2^2 \right)\,{\d}x \\
        &= 
        \eps \int_\Omega \zeta^2 \left( u_2^2 \abs{\nabla \varphi}^2 + u_2 \nabla u_1 \cdot \nabla \varphi + u_1 u_2 \Delta \varphi \right) \,{\d}x \\
        &+ \eps \int_\Omega 2 u_2 \zeta \nabla u_2 \cdot \nabla \zeta \,{\d}x 
        + \sqrt{2}\beta \int_\Omega \zeta^2 \left(\kappa_\star + \frac{1 -\rho_\eps}{\eps} \right) u_2^2\,{\d}x.
    \end{split}
    \end{equation}
    We now proceed to bound the terms on the right-hand side 
    of~\eqref{eq:decay-u2-compu1}. 

    Let $\eta > 0$ be any number. 
    By~\eqref{eq:decay-u2-basic},~\eqref{eq:norm-u}, and  \eqref{eq:varphi-Lp-bounds}, and H\"{o}lder's inequality, we have 
    \begin{equation}\label{eq:decay-u2-compu2}
    \begin{split}
        \int_\Omega \eps \zeta^2 u_2^2 \abs{\nabla \varphi}^2\,{\d}x 
        &\leq \eps\left( \int_B u_2^{2(1+\eta)}\,{\d}x \right)^{\frac{1}{1+\eta}} C_{\frac{2(\eta+1)}{\eta}}(x_0,\,R)^{\frac{\eta}{1+\eta}} \\
        &\leq C_2^{\frac{2}{1+\eta}} C_{\frac{2(\eta+1)}{\eta}}(x_0,\,R)^{\frac{\eta}{1+\eta}} \eps^{1+\frac{1}{1+\eta}},
    \end{split}
    \end{equation} 
    where $C_2$ is the constant in~\eqref{eq:decay-u2-basic} and 
    $C_{\frac{2(\eta+1)}{\eta}}(x_0,\,R)$ is the constant on the 
    right-hand side of~\eqref{eq:varphi-Lp-bounds} for $p = \frac{2(\eta+1)}{\eta}$. 
    
    Next, integrating by parts, we rewrite
    \begin{equation}\label{eq:decay-u2-compu3}
        \eps \int_\Omega \zeta^2\left( u_2 \nabla u_1 \cdot \nabla \varphi + u_1 u_2 \Delta \varphi \right) \,{\d}x = 
        - \eps \int_\Omega \zeta^2 u_1 \nabla u_2 \cdot \nabla \varphi \,{\d}x - \eps \int_\Omega 2 \zeta u_1 u_2 \nabla \zeta \cdot \nabla \varphi \,{\d}x
    \end{equation}
    We bound the right-hand side of~\eqref{eq:decay-u2-compu3} as follows. 
    First, we notice that
    \begin{equation}\label{eq:decay-u2-compu4}
    \begin{split}
        \eps \int_\Omega \zeta^2 \left( u_1 \nabla u_2 \cdot \nabla \varphi \right) \,{\d}x &\leq \int_\Omega \frac{\eps}{2} \zeta^2 \abs{\nabla u_2}^2 \,{\d}x + 2 \eps \int_B \abs{u_1}^2 \abs{\nabla \varphi}^2\,{\d}x \\
        &\leq \int_\Omega \frac{\eps}{2} \zeta^2 \abs{\nabla u_2}^2 \,{\d}x + C_{\beta}(x_0,\,R) \eps,
    \end{split}
    \end{equation}
    where we used Young's inequality, \eqref{eq:norm-u}, and~\eqref{eq:varphi-Lp-bounds}. 
    Next, by H\"{o}lder's inequality,~\eqref{eq:norm-u},~\eqref{eq:decay-u2-basic}, 
    and~\eqref{eq:varphi-Lp-bounds}, 
    \begin{equation}\label{eq:decay-u2-compu5}
    \begin{split}
        \eps \int_\Omega 2 \zeta u_1 u_2 \nabla \zeta \cdot \nabla \varphi \,{\d}x 
        &= 2 \eps \int_B u_1 u_2 (\zeta \nabla \zeta \cdot \nabla \varphi)\,{\d}x \\ 
        &\leq C_\beta \eps \left( \int_B u_2^2 \,{\d}x \right)^{1/2}  \left( \int_B \abs{\nabla \zeta}^2 \abs{\nabla \varphi}^2 \,{\d}x \right)^{1/2} \\
        &\leq C_\beta(x_0,\,R) \frac{\eps^\frac{3}{2}}{R}.
    \end{split}
    \end{equation}
    In both~\eqref{eq:decay-u2-compu4} and~\eqref{eq:decay-u2-compu5}, 
    the constant $C_\beta(x_0,\,R)$ depends only on $x_0$, $R$, $\Omega$, $\Cpot$, 
    and the $L^1(\partial\Omega)$- and the $L^2(\partial \Omega)$-norm of 
    $\Qb \times \partial_\ttau \Qb$. 
    For notational convenience, within this proof we avoid writing explicitly 
    the dependence on $\Omega$, $\Cpot$, and the $L^1(\partial\Omega)$- and 
    the $L^2(\partial \Omega)$-norm of $\Qb \times \partial_\ttau \Qb$ of the constants. 
    
    Next, once again by Young's inequality 
    and~\eqref{eq:decay-u2-basic}, we get 
    \begin{equation}\label{eq:decay-u2-compu6}
    \begin{split}
        2 \eps \int_\Omega\left( \zeta \nabla u_2 \right) \cdot \left( u_2 \nabla \zeta \right) \,{\d}x 
        &\leq \int_\Omega \zeta^2 \left( \frac{\eps}{4} \abs{\nabla u_2}^2 \right) \,{\d}x + 16 \eps\norm{\nabla \zeta}_{L^\infty(K)}^2 \int_B u_2^2\,{\d}x \\
        &\leq \int_\Omega \zeta^2 \left( \frac{\eps}{4} \abs{\nabla u_2}^2 \right) \,{\d}x + 256 C_2 \frac{\eps^2}{R^2}
    \end{split}
    \end{equation}
    Finally, by H\"{o}lder's inequality, Lemma~\ref{lemma:improved-decay-rho}, 
    and~\eqref{eq:decay-u2^4-basic}, 
    \begin{equation}\label{eq:decay-u2-compu6-bis}
    \begin{split}
   	 	\int_B \zeta^2 \left(\kappa_\star + \frac{1 -\rho_\eps}{\eps} \right) u_2^2\,{\d}x 
   	 	&\leq \left( \int_B \left( \frac{\rho - 1}{\eps} - \kappa_\star \right)^2 \zeta^2 \,{\d}x \right)^\frac{1}{2} 
   	 	\left( \int_B \zeta^2 u_2^4 \,{\d}x \right)^\frac{1}{2} \\
   	 	&= \mathrm{O}_{\eps \to 0}(\eps)
   	\end{split}
    \end{equation}
    Combining~\eqref{eq:decay-u2-compu2} for $\eta=1$ with~\eqref{eq:decay-u2-compu3},~\eqref{eq:decay-u2-compu4},~\eqref{eq:decay-u2-compu5},~\eqref{eq:decay-u2-compu6}, and~\eqref{eq:decay-u2-compu6-bis}, 
    we obtain
    \begin{equation}\label{eq:decay-u2-step-1}
    \begin{split}
     \int_\Omega &\zeta^2 \left( \frac{\eps}{4} \abs{\nabla u_2}^2 + \frac{1}{\eps}\left( \abs{\u}^2 - 1 + \sqrt{2}\beta + \sqrt{2}\beta \eps \kappa_\star \right) u_2^2 \right)\,{\d}x   \\
     &\leq C_\beta(x_0,\,R) \eps + C_{\beta}(x_0,\,R) \eps^{3/2} + C_\beta(x_0,\,R) \frac{\eps^{3/2}}{R} + 256 C_2 \frac{\eps^2}{R^2} \\
     &\leq C_\beta(x_0,\,R) \eps.
    \end{split}
    \end{equation}
Notice that, up to this point, $B = B(x_0,\,R)$ can be any ball such that    
$B(x_0,\,2R) \subset \Omega \setminus \spt \mu_\star$ and $\beta$ can be any positive number 
(i.e., we have not used $\sqrt{2}{\beta} > 1$ yet). 
\end{step}

\begin{step}[$L^2$-bounds on $\Delta \varphi$]\label{step:bounds-Delta-phi}
    Assume now, in addition to the previous assumptions, that %$K 
    %$B(x_0,\,2R) \csubset \Omega \setminus (\spt \mu_\star \cup \spt \nu_\star)$ or that 
    $\sqrt{2}\beta = 1 + c > 1$. We have 
   % Either way, we have
    \begin{equation}\label{eq:decay-u2-compu7}
        \left( \abs{\u}^2 - 1 + \sqrt{2}\beta + \sqrt{2}\beta \eps \kappa_\star \right) u_2^2 
        \geq c u_2^2 \qquad \mbox{in } B, 
    \end{equation}
    %either by the uniform convergence of $\abs{\u}^2$ to $1 + \sqrt{2}\beta$ in $B_2$ or because $\sqrt{2}\beta - 1 = c > 0$, respectively. 
%    The implicit constant on the right-hand side of~\eqref{eq:decay-u2-compu7}  %can be taken to be the minimum between $c$ and $2\sqrt{2}\beta$ (and so, it depends only on $\beta$). 
    Combined with~\eqref{eq:decay-u2-step-1},~\eqref{eq:decay-u2-compu7} yields
    \begin{equation}\label{eq:decay-u2-intermediate}
        \int_B u_2^2 \,{\d}x \leq C_\beta\left(x_0,\,R \right) \eps^2,
    \end{equation}
    as well as
    \[
        \int_B \left( \eps \abs{\nabla u_2}^2 + \frac{1}{\eps} u_2^2 \right)\,{\d}x \leq C_\beta\left(x_0,\,R \right) \eps.
    \]
    Now, we take advantage of~\eqref{eq:decay-u2-intermediate} to get $L^2(B)$-bounds on $\zeta^2\Delta \varphi$. With such bounds at hand, we can improve on the estimate of the right-hand side of~\eqref{eq:decay-u2-compu3} and, in turn, obtain a better decay for the left-hand side of~\eqref{eq:decay-u2-step-1} in the smaller ball $B'$, see~\eqref{eq:decay-u2-improved}. With~\eqref{eq:decay-u2-improved} at hand, we obtain Lipschitz bounds on $\varphi$ in $B'$, 
    and in turn this implies, as we shall in the next step, a still faster decay of $u_2$ in the ball $B''$.

    To begin with, we recall that $\varphi$ satisfies the equation 
        \begin{equation}\label{eq:varphi}
        -\div(\rho^2 \nabla \varphi) = \frac{\beta \rho }{\sqrt{2}}\frac{u_1 u_2}{\eps} 
        \qquad \mbox{in } B,
    \end{equation}
    which is obtained by noticing that, by~\eqref{eq:EL-Q} and~\eqref{eq:EL-M},  
    \[
        \Q \times \Delta \Q = \div(\Q \times \nabla \Q) 
        = \frac{\eps}{2}\div (\M \times \nabla \M) = 
        \frac{\eps}{2} \M \times \Delta \M 
        = - \frac{\beta}{\eps} \Q\M \times \M
    \]
    and then that
    \[
        \div\left(\frac{1}{2} \Q \times \nabla \Q\right) = 
        \div(\rho^2 \nabla \varphi) 
        \qquad \mbox{and} \qquad 
        \Q\M \times \M
        = \sqrt{2} \rho u_1 u_2 \qquad \mbox{in } B,
    \]
    where the first identity follows from~\eqref{eq:spectral-thm} and~\eqref{eq:neps-phieps-1} while the second is an immediate consequence of the very definition~\eqref{eq:u} of $\u$.
    
    Let us set 
    \[
        g_\eps := \frac{\beta \rho }{\sqrt{2}}\frac{u_1 u_2}{\eps}.
    \]
    Thanks to~\eqref{eq:decay-u2-intermediate}, we see that the sequence $\{g_\eps\}_\eps$ is bounded in $L^2(B)$. Now, since $\varphi$ does not satisfy an homogeneous Dirichlet boundary condition, we \emph{cannot} conclude directly, by standard elliptic regularity and the $W^{1,2}$-bounds on $\rho$ far from $\spt \mu_\star$, that $\Delta \varphi$ is bounded in $L^2(B)$. However, looking at~\eqref{eq:decay-u2-compu3}, we see that we actually need to bound $\zeta^2 \Delta \varphi$, rather than $\Delta \varphi$. 

    In order to bound $\zeta^2 \Delta \varphi$, we observe that the function $\zeta^2 \varphi$ vanishes on $\partial B$ and satisfies 
    \begin{equation}\label{eq:improved-bound-varphi-compu1}
      -\div\left( \rho \nabla\left(\zeta^2 \varphi\right) \right) = g_\eps \zeta^2 - \rho \nabla \varphi \cdot \nabla \zeta^2 - \div\left(\rho \varphi \nabla \zeta^2\right) 
      \qquad \mbox{in } B.  
    \end{equation}
    Since $\{g_\eps\}_\eps$ is bounded in $L^2(B)$ and by~\eqref{eq:varphi-Lp-bounds} 
    and~\eqref{eq:strong-conv-rho-p}, we see that the right-hand side 
    of~\eqref{eq:improved-bound-varphi-compu1} is bounded in $L^2(B)$, uniformly with respect to $\eps$ 
    (and, indeed, by a constant depending only on $x_0$, $R$, $\Omega$, $\Cpot$, and the $L^1(\partial\Omega)$- and the $L^2(\partial \Omega)$-norm of $\Qb \times \partial_\ttau \Qb$), which yields 
    \[
        \norm{\zeta^2 \varphi}_{W^{2,2}(B)} \leq C_\beta(x_0,\,R).
    \]
    In particular,
    \begin{equation}\label{eq:Delta-varphi}
        \norm{\zeta^2 \Delta \varphi}_{L^2(B)} \leq C_\beta(x_0,\,R). 
    \end{equation}
%    Now, if $G \csubset K$ is a simply connected, open set, we can cover $G$ with finitely many balls $B(x_0,\,R) \subset K$ such that $B(x_0,\,2R) \subset K$, so as to obtain
%    \[
%  \norm{\varphi}_{W^{2,2}(G)} \leq C_\beta(R), \qquad 
 % \norm{\Delta \varphi}_{L^2(G)} \leq C_\beta(R). 
 %   \]
\end{step}

\begin{step}[Lipschitz bounds on $\varphi$ in $B'$]\label{step:Lip-bounds-phi}
    Using~\eqref{eq:decay-u2-intermediate} and~\eqref{eq:Delta-varphi}, we obtain
    \begin{equation}\label{eq:decay-u2-compu8}
       \eps \int_B u_1 u_2 \zeta^2 \Delta   \varphi \,{\d}x \leq C_\beta(x_0,\,R) \eps^2.
    \end{equation}
    Let $\eta > 0$ be any number. 
    By H\"{o}lder's inequality with $p_1 = 1/2$, $p_2 = 2 + \eta$, 
    and $p_3 = \frac{2(2+\eta)}{\eta}$ (so that $1/p_2 + 1/p_3 = 1/2$ and $1/p_1 + 1/p_2 + 1/p_3 = 1$), 
    we get
    \begin{equation}\label{eq:decay-u2-compu9}
    \begin{split}
        \eps \int_B \zeta^2 u_2 \nabla u_1 \cdot \nabla \varphi \,{\d}x 
        &\leq \eps \left( \int_B \abs{u_2}^2 \,{\d}x \right)^{\frac{1}{2}} \left( \int_B \abs{\nabla u_1}^{p_2} \,{\d}x \right)^{\frac{1}{p_2}} \left( \int_B \abs{\nabla \varphi}^{p_3} \,{\d}x \right)^{\frac{1}{p_3}} \\
        &\leq C_\beta(x_0,\,R) C_{\beta,\eta}^{\frac{\eta}{2+\eta}} \,\eps^{1 + \frac{1}{2 +\eta}} \\
        &= C_{\beta,\eta}(x_0,\,R) \, \eps^{1+\frac{1}{2+\eta}}.
    \end{split}
    \end{equation}
    Thus, using~\eqref{eq:decay-u2-intermediate},~\eqref{eq:Delta-varphi}, 
    and~\eqref{eq:decay-u2-compu8}, we may improve on~\eqref{eq:decay-u2-step-1} so as to obtain
    \begin{equation}\label{eq:decay-u2-improved}
        \int_{B'}\left( \eps \abs{\nabla u_2}^2 + \frac{1}{\eps}\left( \abs{\u}^2 - 1 + \sqrt{2}\beta + \sqrt{2}\beta \eps \kappa_\star \right)u_2^2\right)\,{\d}x \leq C_{\beta,\gamma}(x_0,\,R) \eps^{1+\frac{\gamma}{2}}, 
        %\qquad \forall \gamma \in (0,\,1),
    \end{equation}
    where $\gamma = \frac{1}{2+\eta}$. Since $\eta > 0$ was arbitrary, 
    this implies, in particular,
    \begin{equation}\label{eq:decay-u2-improved-bis}
        \int_{B'} \left( \frac{u_2}{\eps} \right)^2 \,{\d}x 
        \leq C_{\beta,\gamma}(x_0,\,R) \eps^{\frac{\gamma}{2}}, 
        \qquad \forall \gamma \in (0,\,1)
    \end{equation}
    In turn,~\eqref{eq:decay-u2-improved-bis} implies that
    \begin{equation}\label{eq:decay-L2-norm-geps}
        \norm{g_\eps}_{L^2(B')} \leq C_{\beta,\gamma}(x_0,\,R) \eps^\frac{\gamma}{4}, 
        \qquad \forall \gamma \in (0,\,1)
    \end{equation}
    Plugged in~\eqref{eq:varphi} and along 
    with~\eqref{eq:improved-bound-varphi-compu1},~\eqref{eq:decay-u2-improved-bis} 
    allows to conclude that 
    \begin{equation}\label{eq:W22+est-varphi}
        \norm{\zeta^2 \varphi}_{W^{2,2+\frac{\gamma}{2}}(B')} \leq C_{\beta,\gamma}(x_0,\,R), 
        \qquad \forall \gamma \in (0,\,1),
    \end{equation}
    so that, by Sobolev embedding, 
    \begin{equation}\label{eq:Lip-bound-varphi}
        \norm{\nabla \varphi}_{L^\infty(B')} \leq C_\beta(x_0,\,R).
    \end{equation}
    Notice that the constant on the right-hand side of~\eqref{eq:Lip-bound-varphi} 
    does not depend on $\gamma$ because, since~\eqref{eq:W22+est-varphi} 
    holds for any $\gamma \in (0,\,1)$, we can take the constant on the 
    right-hand side of~\eqref{eq:Lip-bound-varphi} to be the infimum 
    over the constants $C_{\beta,\gamma}(x_0,\,R)$ obtained by 
    combining~\eqref{eq:W22+est-varphi} and Sobolev embedding.
\end{step}

\begin{step}[Improved decay of $u_2$ in $B''$]\label{step:u2-improved-decay}
    In order to take advantage 
    of~\eqref{eq:Lip-bound-varphi}, we test~\eqref{eq:EL-u2} against 
    $\eps \psi^2 u_2$, where $\psi \in C^\infty(\R^2)$ is any cut-off function such that
    \begin{gather*}
        0 \leq \psi \leq 1, \qquad 
        \psi \equiv 1 \quad \mbox{in } B'', \qquad 
        \spt\psi \subset B', \\ 
        \abs{\nabla \psi} \leq 4\left(\frac{2}{R}\right), \qquad \abs{\nabla^2 \psi} \leq 4\left(\frac{4}{R^2}\right) \qquad \mbox{in } \R^2. 
    \end{gather*}
    Then, upon using~\eqref{eq:Lip-bound-varphi} and~\eqref{eq:decay-u2-intermediate}, 
    we obtain the improved estimate
    \begin{equation}\label{eq:decay-u2-compu10}
    \begin{split}
       \eps \int_{B'} \psi^2 u_2 \nabla u_1 \cdot \nabla \varphi \,{\d}x &\leq C_\beta(x_0,\,R)\,\sqrt{\eps}\left( \int_B u_2^2 \,{\d}x \right)^{\frac{1}{2}} \left( \eps \int_B \abs{\nabla u_1}^2 \,{\d}x\right)^{\frac{1}{2}} \\
       &\leq C_\beta(x_0,\,R) \eps^{3/2}.
    \end{split}   
    \end{equation}
    Notice that the integral on the left-hand side is over the smaller 
    ball $B'$ rather than over the original ball $B$.
    In turn, since $\psi \equiv 1$ in $B''$, we obtain the improved decay 
    \begin{equation}\label{eq:decay-u2-improved-ter}
        \int_{B''} \left( \eps \abs{\nabla u_2}^2 + \frac{1}{\eps}\left( \abs{\u}^2 - 1 + \sqrt{2}\beta \right) u_2^2 \right)\,{\d}x 
        \leq C_\beta(x_0,\,R) \, \eps^\frac{3}{2}.
    \end{equation}

    \begin{remark}
        Instead of~\eqref{eq:decay-u2-improved}, we could employ~\eqref{eq:decay-u2-improved-bis} in the estimate~\eqref{eq:decay-u2-compu10}. This would give a factor $\eps^\frac{3+\gamma}{2}$ on the right-hand side of~\eqref{eq:decay-u2-improved-ter}, rather than $\eps^\frac{3}{2}$, making the decay of $u_2$ slightly faster. However, as can be easily checked, since $\gamma \in (0,\,1)$, this additional speed of convergence is not enough, according to the iteration argument presented in Step~\ref{step:decay-u2-iteration} below, to yield, after iteration, a further improved decay of $u_2$ with respect to the one achieved by starting from~\eqref{eq:decay-u2-improved-ter}. For this reason, we preferred to keep things as simple as possible, avoiding the use of~\eqref{eq:decay-u2-improved-bis} and so the introduction of a further dependence on $\gamma$ in~\eqref{eq:decay-u2-improved-ter}.
    \end{remark}
\end{step}

\begin{step}[Iteration]\label{step:decay-u2-iteration}
    Recall once again that we assume %at least one of the two conditions 
%    $B(x_0,\,2R) \csubset \Omega \setminus (\spt \mu_\star \cup \spt \nu_\star)$ or  
	$\sqrt{2}\beta = 1 + c > 1$. Under %these assumptions,
	this assumption,~\eqref{eq:decay-u2-improved-ter} implies that
    \begin{equation}\label{eq:decay-u2-iteration-1}
        \int_{B_{1/4}} u_2^2\,{\d}x \leq C_\beta(x_0,\,R) \eps^\frac{5}{2}.
    \end{equation}
    With~\eqref{eq:decay-u2-iteration-1} at hand, we repeat the argument from Step~\ref{step:u2-basic-bounds} to Step~\ref{step:u2-improved-decay}, but with 
    \[
        B = B_{1/4}, \qquad B' = B_{1/8}, \qquad B'' = B_{1/16}.
    \]
    Correspondingly, we take a new cut-off function $\zeta \in C^\infty_c(\R^2)$ still as in~\eqref{eq:cut-off}, but with~\eqref{eq:cut-off-grad} replaced by
    \[
        \abs{\nabla \zeta} \leq \frac{16}{R}, 
        \qquad \abs{\nabla^2 \zeta} \leq \frac{32}{R^2}, \qquad \mbox{in } \R^2.
    \]
    In general, at the $n$-th iteration step, we will take 
        \[
        B = B_{2^{-2n}}, \qquad B' = B_{2^{-2n+1}}, \qquad B'' = B_{2^{-2n-2}} .
    \]
    and $\zeta = \zeta_n$ such that~\eqref{eq:cut-off} holds, as well as
    \begin{equation}\label{eq:cut-off-grad-n}
        \abs{\nabla \zeta} \leq 2^n\left(\frac{2}{R}\right), 
        \qquad \abs{\nabla^2 \zeta} \leq 2^n\left(\frac{4}{R^2}\right), \qquad \mbox{in } \R^2.
    \end{equation}
    At the first iteration, we obtain that~\eqref{eq:decay-u2-compu2},~\eqref{eq:decay-u2-compu6},~\eqref{eq:decay-u2-compu8}, and~\eqref{eq:decay-u2-compu9} can be replaced, respectively, as follows. First, by~\eqref{eq:decay-u2-improved-ter} and~\eqref{eq:Lip-bound-varphi}, 
    \begin{equation}\label{eq:decay-u2-compu11}
        \eps \int_B \zeta^2 u_2^2 \abs{\nabla \varphi}^2 \,{\d}x \lesssim \eps^\frac{7}{2},
    \end{equation}
    By H\"{o}lder's inequality,~\eqref{eq:Delta-varphi},~\eqref{eq:decay-L2-norm-geps}, and~\eqref{eq:norm-u}, we have
    \begin{equation}\label{eq:decay-u2-compu12}
        \eps \int_B u_1 u_2 \zeta^2 \Delta \varphi\,{\d}x \lesssim \eps^2.
    \end{equation}
    By H\"{o}lder's inequality,~\eqref{eq:decay-u2-iteration-1}, 
    and~\eqref{eq:decay-u2-improved-ter}, we have 
    \begin{equation}\label{eq:decay-u2-compu13}
        \eps \int_B (u_2 \nabla \zeta)\cdot (\zeta \nabla u_2) \,{\d}x \lesssim \eps^\frac{5}{2}. 
    \end{equation}
    Finally, H\"{o}lder's inequality,~\eqref{eq:Lip-bound-varphi}, 
    and~\eqref{eq:decay-u2-iteration-1} yield 
    \begin{equation}\label{eq:decay-u2-weakest-improv}
        \eps \int_B \zeta^2 u_2 \nabla u_1\cdot \nabla\varphi\,{\d}x \lesssim \eps^\frac{7}{4}. 
    \end{equation}
    In the above inequalities, the implicit constant on the right-hand side depends only on $x_0$, $R$, $\beta$, $\Omega$, $\Cpot$, and the $L^1(\partial \Omega)$- and the $L^2(\partial \Omega)$-norms of $\Qb \times \partial_\ttau \Qb$. 
    
    Since the decay on the right-hand side of~\eqref{eq:decay-u2-weakest-improv} is the slowest (both at this step and at any other step), it provides the leading order in the improvement of the decay of the left-hand side of~\eqref{eq:decay-u2-improved-ter}. 
    More precisely, combining~\eqref{eq:decay-u2-compu11},~\eqref{eq:decay-u2-compu12},~\eqref{eq:decay-u2-compu13}, and~\eqref{eq:decay-u2-weakest-improv} with~\eqref{eq:decay-u2-compu1}, we see that, after the first iteration step, we obtain  
    \begin{equation}\label{eq:decay-u2-improved-4}
                \int_{B_{1/16}} \left( \eps \abs{\nabla u_2}^2 + \frac{1}{\eps}\left( \abs{\u}^2 - 1 + \sqrt{2}\beta \right) u_2^2 \right)\,{\d}x 
        \leq C_\beta(x_0,\,R) \, \eps^\frac{7}{4}, 
    \end{equation}
    as well as 
    \[
        \int_{B_{1/16}} \left( \eps \abs{\nabla u_2}^2 + \frac{1}{\eps} u_2^2 \right)\,{\d}x 
        \leq C'_\beta(x_0,\,R) \, \eps^\frac{7}{4}.
    \]    
    By a trivial induction argument, we see that at the $n$-th step we have
    \begin{equation}\label{eq:decay-u2-improved-n}
        \int_{B_{2^{-2n-2}}} \left( \eps \abs{\nabla u_2}^2 + \frac{1}{\eps}\left( \abs{\u}^2 - 1 + \sqrt{2}\beta \right) u_2^2 \right)\,{\d}x 
        \leq C_{\beta,n}(x_0,\,R) \, \eps^{\alpha_n},
    \end{equation}
    where the dependence on $n$ arises because of the dependence on $n$ in~\eqref{eq:cut-off-grad-n}, and
    \begin{equation}\label{eq:decay-u2-improved-n-bis}
        \int_{B_{2^{-2n-2}}} \left( \eps \abs{\nabla u_2}^2 + \frac{1}{\eps} u_2^2 \right)\,{\d}x 
        \leq C'_{\beta,n}(x_0,\,R) \, \eps^{\alpha_n}. 
    \end{equation}
    The exponent $\alpha_n$ is given by the iterative formula
    \[
    	\alpha_{n+1} = \frac{\alpha_n + 2}{2}, \qquad \alpha_0 = 1,
    \]
    i.e., 
    \begin{equation}\label{eq:alpha_n}
        \alpha_n := 2-2^{-(n+1)}.
    \end{equation}
\end{step}

\begin{step}[Covering argument and conclusion]
    Let $K \subset \Omega \setminus \spt\mu_\star$ be a compact set. 
    %. Suppose in addition that $1+\sqrt{2}\beta \geq c > 0$ 
    %or that $K \subset \Omega \setminus (\spt\mu_\star \cup \spt \nu_\star)$. 
    Choose $\alpha \in (0,\,3)$ arbitrarily and let $G \csubset K$ be any simply connected, open set with smooth boundary. 
    Let $n_\alpha \in \mathbb{N}$ be the smallest index such that $\alpha \leq \alpha_{n_\alpha}$, where $\alpha_{n_\alpha}$ is given by~\eqref{eq:alpha_n} for $n = n_\alpha$. 
    Then, by~\eqref{eq:decay-u2-improved-n}, we have 
    \[
            \int_{B_{2^{-2{n_\alpha}-2}}} \left( \eps \abs{\nabla u_2}^2 + \frac{1}{\eps} u_2^2 \right)\,{\d}x 
        \leq C'_{\beta,n_\alpha}(x_0,\,R) \, \eps^{\alpha}.
    \]
    Moreover, we can cover $G$ with finitely many balls of radius $2^{{-2 n_\alpha} -2}R$, for $R > 0$ sufficiently small (depending only on $n_\alpha$ and on $\dist(G,\,K)$) that the corresponding concentric balls of radius $2^{-2n_\alpha-1}R$ are still contained in $K$. As a consequence, we obtain that~\eqref{eq:decay-u2-goal1} and~\eqref{eq:decay-u2-goal2} hold, for constants $C_\alpha(\dist(G,\,K))$, $C_\alpha'(\dist(G,\,K))$ depending only on $\dist(G, \partial K)$, $\alpha$, $\beta$, $\Omega$, $\Cpot$, and the $L^1(\partial\Omega)$- and the $L^2(\partial \Omega)$-norm of $\Qb \times \partial_\ttau \Qb$.
\end{step}
\end{proof}

\begin{remark}\label{rk:decay-away-from-singular-sets}
	As it can be checked immediately by inspection of the proof, due to 
	the local uniform convergence $\abs{\u}^2 \to 1 + \sqrt{2}\beta$ 
	in $\Omega \setminus (\spt\mu_\star \cup \spt\nu_\star)$ (which is 
	an obvious consequence of item~\ref{item:unif-conv-M} of Theorem~\ref{thm:B-CDS}), 
	Proposition~\ref{prop:decay-u2} holds without restrictions 
	on $\beta$ if we assume that 
	$K \subset \Omega \setminus (\spt\mu_\star \cup \spt\nu_\star)$.
\end{remark}

\begin{corollary}\label{cor:unif-conv-u2}
	Let $\{(\Q_\eps,\,\M_\eps)\}_\eps$ be a sequence of critical points of 
	$\F_\eps$ satisfying boundary conditions either as in \eqref{hp:bc-dir} or as 
	in~\eqref{hp:bc-mixed}, as well as assumption~\eqref{hp:equibdd-f}. 
    Assume, in addition, that $\sqrt{2}\beta > 1$. 
 %   or, alternatively, assume that $K \subset \Omega \setminus (\spt\mu_\star \cup \spt\nu_\star)$.
    Let $K \subset \Omega \setminus \spt\mu_\star$ be a compact set.  
    and let $G \csubset K$ be any simply connected, open set with smooth boundary.   
    Let $\{\u_\eps\}_\eps$ be the sequence of maps defined in $G$ as prescribed by~\eqref{eq:u}. Then,
    \begin{equation}\label{eq:unif-decay-u2}
    	\norm{u_{2,\eps}}_{L^\infty(G)} \lesssim \eps^{\gamma/2}
    \end{equation}
    for every $\gamma \in (0,\,1)$. In particular, as $\eps \to 0$, 
    \begin{equation}\label{eq:unif-conv-u2}
    	u_{2,\eps} \to 0 \quad \mbox{uniformly in } K.
    \end{equation}
\end{corollary}

\begin{proof}
	Since $G$ has smooth boundary, it follows from Morrey's inequality 
	that
	\[
		\norm{u_{2,\eps}}_{L^\infty(G)} \lesssim \norm{u_{2,\eps}}_{W^{1,2}(G)}. 
	\]
	Then,~\eqref{eq:unif-decay-u2} follows immediately from~\eqref{eq:decay-u2-goal2}. 
	Now, for any compact set 
	$K \subset \Omega \setminus \spt\mu_\star$, we can find another compact 
	$K'$ and a simply connected, open set $G$ with smooth boundary such 
	that $K \subset G \csubset K'$. Then,~\eqref{eq:unif-decay-u2} holds 
	in $G$ and implies $\norm{u_{2,\eps}}_{L^\infty(K)} \lesssim \eps^{\gamma}$ for 
	any $\gamma \in (0,\,1)$, so that~\eqref{eq:unif-conv-u2} follows 
	immediately by taking the limit as $\eps \to 0$. 
\end{proof}

\begin{remark}
	Besides specifying the rate of convergence, the point 
	of Corollary~\ref{cor:unif-conv-u2} is that the component 
	$\{u_{2,\eps}\}$ converges uniformly to zero locally everywhere 
	away from $\spt \mu_\star$, and so even across 
	the energy concentration set for the $\M$-component, i.e., 
	even on $\spt \nu_\star$.
\end{remark}

We now combine Proposition~\ref{prop:decay-u2} and Corollary~\ref{cor:unif-conv-u2} 
so as to obtain an improved decay for $u_{2,\eps}^4$, which will be useful 
in Section~\ref{sec:integrality}. 
\begin{corollary}\label{cor:decay-u2^4}
    Let $\{(\Q_\eps,\,\M_\eps)\}_\eps$ be a sequence of critical points of 
	$\F_\eps$ satisfying boundary conditions either as in \eqref{hp:bc-dir} or as 
	in~\eqref{hp:bc-mixed}, as well as assumption~\eqref{hp:equibdd-f}. 
	Assume, in addition, that $\sqrt{2}\beta > 1$.  
    Let $K \subset \Omega \setminus \spt\mu_\star$ be any compact set   
    %or, alternatively, assume that $K \subset \Omega \setminus (\spt\mu_\star \cup \spt\nu_\star)$.
    and let $G \csubset K$ be any simply connected, open set, with smooth boundary $\partial G$.   
    Let $\{\u_\eps\}_\eps$ be the sequence of maps defined in $G$ as prescribed by~\eqref{eq:u}. Then, for any $\gamma \in (0,\,1)$ and any $\eps$ small enough, depending only on $K$ and $\beta$, we have
    \begin{equation}\label{eq:decay-u2-goal3}
        \int_G u_{2,\eps}^4 \,{\d}x 
        \leq C_\gamma(\dist(G,\,\partial K)) \eps^{4\gamma}.
    \end{equation}
    where $C_\gamma(\dist(G,\,\partial K))$ depends only on 
    $\dist(G, \partial K)$, $\gamma$, $\beta$, $\Omega$, $\Cpot$, 
    and the $L^1(\partial\Omega)$- and the $L^2(\partial\Omega)$-norms of $\Qb$. 
\end{corollary}

\begin{proof}
	Let $K$ and $G$ be as in the statement. Then, by Corollary~\ref{cor:unif-conv-u2}, 
	we have 
	\[
		u_{2,\eps}^2 \lesssim \eps^\gamma, \qquad \forall \gamma \in (0,\,1).
	\] 
	Thus, 
	\[
		\int_G u_{2,\eps}^4 \,{\d}x \lesssim \eps^\gamma \int_G u_{2,\eps}^2 \,{\d}x,   
	\]
	and the conclusion follows from~\eqref{eq:decay-u2-goal2}.
\end{proof}

%\begin{proof}
%    Let $K \subset \Omega \setminus \spt\mu_\star$ be any compact set and 
%    assume that $\sqrt{2}\beta > 1$. 
%    By the clearing-out result for $\rho_\eps$, we know that there exists 
%    $\bar{\eps} := \bar{\eps}(K,\,\beta)$ such that we have 
%    \[
%        \rho_\eps \geq \frac{1}{\sqrt{2}\beta} \qquad \mbox{in } K,
%    \]
%    for any $0 < \eps \leq \bar{\eps}$. But then we have
%    \[
%        \abs{\u_\eps}^2 - 1 + \sqrt{2}\beta \rho_\eps 
%        = u_{1,\eps}^2 + u_{2,\eps}^2 - 1 + \sqrt{2}\beta \rho_\eps 
%        \geq u_{2,\eps}^2 \qquad \mbox{in } K,
%    \]
%    so that the conclusion follows from~\eqref{eq:decay-u2-goal1}.
%\end{proof}

%
%\begin{remark}\label{rk:u2-no-concentration}
%    The main outcome of Proposition~\ref{prop:decay-u2} is that 
%    the component $u_{2,\eps}$ \emph{cannot} concentrate energy 
%    in the regime $\sqrt{2} \beta > 1$. 
%    Consequently, as we are going to see in Section~\ref{sec:integrality}, 
%    it does not contribute to the limiting varifold $\bbV_\star$. 
%\end{remark}

%------------------------------------------------
\section{Proof of Theorem~\ref{mainthm:integrality}}\label{sec:integrality}

In this section, we prove Theorem~\ref{mainthm:integrality}.
%\begin{theorem}\label{thm:integrality}
%    Let $\zeta_\star = \mathfrak{v}_\star\,\H^1 \mres \mathfrak{S}_\star$ 
%    be the limiting potential energy density and let 
%    $\bbV_\star = \v(\mathfrak{S}_\star,\,\mathfrak{v}_\star)$ be the 
%    varifold associated with $\zeta_\star$ according to Theorem~\ref{thm:B-CDS}. 
%    Let 
%    \[
%        \sigma_\beta := 
%        %\frac{1}{\sqrt{1+\sqrt{2}\beta}}
%        %\int_{-\sqrt{1+\sqrt{2}\beta}}^{\sqrt{1+\sqrt{2}\beta}} \sqrt{\frac{h_1(s)}{2}}\,{\d}s =
%        \frac{\sqrt{2}}{3}\left( 1 + \sqrt{2}\beta \right)^{3/2},
%    \]
%    where $h_1$ is the function defined in~\eqref{eq:h1-h2}. Assume that 
%    $\sqrt{2}\beta > 1$. 
%    Then, the varifold 
%    \[
%    	\sigma_\beta^{-1} \bbV_* := \left( \mathfrak{S_*},\,\frac{\mathfrak{v}_\star}{ \sigma_\beta} \right)
%    \]
%    has integer multiplicity.
%\end{theorem}

Before going to the details of the proof, we 
recall that, if $\sqrt{2}\beta > 1$, then the function $h_2$ 
in~\eqref{eq:h1-h2} is non-negative, for every value of $u_1$, $u_2$. 
Thus, if   
$B = B(x_0,\,R) \csubset \Omega \setminus \spt\mu_\star$ is any ball 
and $\{ (\Q_\eps,\,\M_\eps) \}_\eps$ 
is any sequence of critical points of 
$\F_\eps$, then, by Proposition~\ref{prop:decay-u2} and~\eqref{eq:norm-u}, 
we have 
\begin{equation}\label{eq:decay-h2}
    \frac{1}{\eps}\int_B h_2(u_{1_\eps}, u_{2,\eps}) \,{\d}x  
    = {\rm O}_{\eps \to 0}(\eps^{\alpha}) 
    \qquad \mbox{for all } \alpha \in (0,\,2).
\end{equation} 
In analogy with the definition~\eqref{eq:zeta-eps}, and taking~\eqref{eq:decay-h2} 
into account, we set 
\[
	\zeta_\eps^{(h)} := \frac{1}{\eps} h(u_{1,\eps}), \qquad 
	\zeta_\eps^{(h_1)} := \frac{1}{\eps} h_1(u_{1,\eps}).
\]
By~\eqref{eq:ACeps-equibdd}, it then follows that 
\[
	\zeta_\eps^{(h)} \rightharpoonup^* \zeta_\star^{(h)}, \qquad 
	\zeta_\eps^{(h_1)} \rightharpoonup^* \zeta_\star^{(h_1)} 
	\qquad \mbox{as } \eps \to 0.
\]
On the other hand, by~\eqref{eq:decay-h2}, we have  
\[
	\int_K \abs{\zeta_\eps^{(h)} - \zeta_\eps^{(h_1)}} \,{\d}x 
	\to 0 \qquad \mbox{as } \eps \to 0,
\]
for any compact set $K \subset \Omega \setminus \spt\mu_\star$, 
and thus the limiting varifold associated with the \emph{vectorial} potential 
potential $h$ is integral if and only if the one associated with 
the \emph{scalar} potential $h_1$ is. On the other hand, we already know 
that $\zeta_\star^{(h)} = \zeta_\star$ in $\Omega \setminus \spt\mu_\star$. 
In the proof of Theorem~\ref{mainthm:integrality} below, we prove 
that the density $\mathfrak{h}_\star$ of $\zeta_\star^{(h_1)}$ 
exists at every point $x_0$ in $\Omega \setminus \spt \mu_\star$ 
and that $\sigma_\beta^{-1} \mathfrak{h}_\star(x_0)$ is an integer. 
Then, we conclude that $\sigma_\beta^{-1} \mathfrak{v}_\star(x_0)$ 
takes only integer values by showing that the difference 
$\zeta_\eps - \zeta_\eps^{(h_1)}$ vanishes in $L^1(K)$ as $\eps \to 0$, 
for every $K \subset \Omega \setminus \spt \mu_\star$,   
and using that $\zeta_\star(\spt \mu_\star) = 0$. (This last property 
follows from the monotonicity of $\zeta_\star$ --- 
see \cite[Proposition~3.18(ii)]{CDS2}.)

\begin{proof}[{Proof of Theorem~\ref{mainthm:integrality}}]
	The equation satisfied by the component $u_{1,\eps}$ is obtained 
	by projecting~\eqref{eq:EL-u} along the direction $\e_1 := (1,\,0)^{\rm t}$ and reads as follows:
    \begin{equation}\label{eq:EL-u1}
    \begin{split}
        -\eps \Delta u_{1,\eps} &+ \frac{1}{\eps}\left(u_{1,\eps}^2 - 1 - \sqrt{2}\beta \rho_\eps \right)u_{1,\eps} \\
        &= \eps \left( -u_{1,\eps} \abs{\nabla \varphi_\eps}^2 - \nabla u_{2,\eps} \cdot \nabla \varphi_\eps - u_{2,\eps} \Delta \varphi_\eps \right) - \frac{1}{\eps}u_{2,\eps}^2 u_{1,\eps}.
    \end{split}
    \end{equation}
    This equation holds pointwise in any simply 
    connected, open set $G \csubset \Omega \setminus \spt\mu_\star$ 
    with smooth boundary $\partial G$. Since $\sigma_\beta$ is a constant 
    which depends only on $\beta$, the integrality claim in the 
    statement is local 
    in nature, and therefore we can localise the argument. More precisely, 
    the strategy goes as follows. First, working in balls 
    $B = B(x_0,\,R) \csubset \Omega \setminus \spt\mu_\star$ and relying on the 
    results in Section~\ref{sec:decay-u2}, we recast~\eqref{eq:EL-u1} 
    in the form considered~\cite{NagaseTonegawa}, see~\eqref{eq:EL-u1-bis} 
    below, obtaining at the same time the required decay estimates for the perturbation 
    term on the right-hand side.  
    This allows us to apply the machinery in~\cite{NagaseTonegawa} in any compact set 
    $K \subset \Omega \setminus \spt\mu_\star$.
    Then, from~\cite{NagaseTonegawa} we infer that 
    $\mathfrak{v}_\star(x_0)$ is an integer multiple of $\sigma_\beta$ for 
    any $x_0 \in K$, so that the conclusion follows by letting 
    $K$ vary in $\Omega \setminus \spt\mu_\star$ and recalling that 
    $\mathfrak{v}_\star = 0$ on $\spt \mu_\star$.
    
%    We now proceed to rewrite~\eqref{eq:EL-u1} in a suitable form 
%    to apply the integrality result in \cite{NagaseTonegawa}.
    
    \setcounter{step}{0}
    \begin{step}[Rewriting~\eqref{eq:EL-u1}]
	The aim of this step is to rewrite~\eqref{eq:EL-u1} in the form
    \begin{equation}\label{eq:EL-u1-bis}
    	-\eps \Delta \widetilde{u}_\eps + \frac{1}{\eps}\left( \widetilde{u}_\eps^2 - 1 - \sqrt{2}\beta \right)\widetilde{u}_\eps 
    	= F_\eps\left(\widetilde{u}_\eps,\,u_{2,\eps}\right),
    \end{equation}
    where
    \begin{equation}\label{eq:widetilde-u}
		\widetilde{u}_\eps := \lambda_\eps u_{1,\eps}, \qquad 
		\lambda_\eps := \left(\frac{1 + \sqrt{2}\beta}{1+\sqrt{2\beta} + \eps \sqrt{2}\beta \kappa_\star}\right)^{1/2} ,
    \end{equation}
    and $F_\eps\left( \widetilde{u}_\eps,\,u_{2,\eps} \right)$, 
    defined in~\eqref{eq:capitolF} below, 
    satisfies the local estimate
   	\begin{equation}\label{eq:decay-capitolF}
		\frac{1}{\eps} \int_K F_\eps^2\,{\d}x \leq C_{\beta,K},
   	\end{equation}
   	in any compact set $K \subset \Omega \setminus \spt\mu_\star$, 
   	where $C_{\beta,K}$ is a constant depending only on $\beta$, $K$, 
   	$\Omega$, $\Cpot$, and the $L^1(\partial \Omega)$- and 
   	the $L^2(\partial \Omega)$-norm of $\Qb \times \partial_\ttau \Qb$.
   	
   	\medskip
    \noindent
    Towards the announced purpose, we start by setting (for ease of notation, 
    we drop the subscripts $\eps$ in the intermediate formulae below, writing 
    $u_1$, $u_2$, $\varphi$ in place of $u_{1,\eps}$, $u_{2,\eps}$, $\varphi_\eps$)
    \begin{align*}
        g_\eps(u_1,\,u_2) &:= -\eps\left(u_1 \abs{\nabla \varphi}^2 + \nabla u_2 \cdot \nabla \varphi + u_2 \Delta \varphi\right), 
        \\
        f_\eps &:= g_\eps - \frac{u_2^2 u_1}{\eps},
    \end{align*}
    which leads to
   \begin{equation}\label{eq:u1-pert}
        -\eps \Delta u_1 + \frac{1}{\eps}\left(u_1^2 - 1 - \sqrt{2}\beta \rho \right)u_1 = f_\eps(u_1,\,u_2).
    \end{equation}
    We claim that, for any ball $B = B(x_0,\,R) \csubset \Omega \setminus \spt\mu_\star$, 
    \begin{equation}\label{eq:decay-f}
        \frac{1}{\eps}\int_B f_\eps^2\,{\d}x \leq C_\beta(x_0,\,R) \eps^\gamma, 
        \qquad \forall \gamma \in (0,\,1),
    \end{equation}
    for a constant $C_\beta(x_0,\,R)$  
    depending only $\beta$, $x_0$, $R$, $\Omega$, $\Cpot$, 
    and the $L^1(\partial \Omega)$- and the $L^2(\partial \Omega)$-norm 
    of $\Qb \times \partial_\ttau \Qb$.
    
	In order to prove the claim, we first observe that it follows 
	from~\eqref{eq:decay-u2-goal3} and~\eqref{eq:norm-u} that 
    \begin{equation}\label{eq:decay-u2^4}
        \frac{1}{\eps}\int_B \frac{u_2^4 u_1^2}{\eps^2}\,{\d}x \leq 
        C_\beta(x_0,\,R) \eps^\gamma, \qquad \forall \gamma \in (0,\,1),
    \end{equation}
    for every $\eps$ small enough, depending on $K$ and $\beta$ only. 
    In order to bound $g_\eps$, we observe that 
    by an argument similar to the one in 
    Step~\ref{step:bounds-Delta-phi}, Step~\ref{step:Lip-bounds-phi}, 
    and Step~\ref{step:decay-u2-iteration}
    in the proof of Proposition~\ref{prop:decay-u2}, we may assume that 
    \begin{gather*}
        \norm{\nabla \varphi}_{L^\infty(B)} \leq C_\beta(x_0,\,R), \qquad 
        \norm{\Delta \varphi}_{L^2(B)} \leq C_\beta(x_0,\,R),  \\ 
        \norm{\nabla u_2}_{L^2(B)} \leq  C_{\beta,\gamma}(x_0,\,R) \eps^{\gamma}, 
        \qquad \forall \gamma \in (0,\,1),
    \end{gather*}
    for constants $C_\beta(x_0,\,R)$, $C_{\beta,\gamma}(x_0,\,R)$ 
    depending only $\beta$, $\gamma$, $x_0$, $R$, $\Omega$, $\Cpot$, 
    and the $L^1(\partial \Omega)$- and the $L^2(\partial \Omega)$-norm 
    of $\Qb \times \partial_\ttau \Qb$. From these bounds, it follows that
    \begin{equation}\label{eq:decay-geps}
        \norm{g_\eps}_{L^2(B)} \leq C_\beta(x_0,\,R)\,\eps.
    \end{equation}
    From~\eqref{eq:decay-u2^4} and~\eqref{eq:decay-geps},~\eqref{eq:decay-f} follows.
    
	Now, we take care of the fact that, as written, the potential term 
	on the right-hand side of~\eqref{eq:u1-pert} depends on $\rho$ 
	(and thus implicitly on both $\eps$ and $x$). To this purpose,
	we further rewrite~\eqref{eq:u1-pert} as follows:
    \begin{equation}\label{eq:u1-pert-bis}
    -\eps \Delta u_1 + \frac{1}{\eps}\left(u_1^2 - 1 - \sqrt{2}\beta - \eps \sqrt{2}\beta \kappa_\star \right)u_1 
    = \sqrt{2}\beta \left( \frac{\rho-1}{\eps} - \kappa_\star \right) u_1 + f_\eps(u_1,\,u_2).
	\end{equation}
	Notice that, by Lemma~\ref{lemma:improved-decay-rho}, 
	\begin{equation}\label{eq:decay-rho-pot}
   \frac{1}{\eps} \int_B \left( \frac{\rho-1}{\eps} - \kappa_\star \right)^2 u_1^2 \,{\d}x \lesssim 
   \frac{1}{\eps} \int_B \left( \frac{\rho-1}{\eps} - \kappa_\star \right)^2 \,{\d}x
   \leq C_\beta(x_0,\,R).
	\end{equation}
	Next, as anticipated in~\eqref{eq:widetilde-u}, we set 
	\[
		\widetilde{u}_\eps := \lambda_\eps u_{1,\eps}, 
		\qquad
		\lambda_\eps := \left(\frac{1 + \sqrt{2}\beta}{1+\sqrt{2\beta} + \eps \sqrt{2}\beta \kappa_\star}\right)^{1/2}.
	\]
	Rewriting~\eqref{eq:u1-pert-bis} in terms of $\widetilde{u} = \widetilde{u}_\eps$, 
    we obtain
    \begin{equation}\label{eq:u1-pert-ter}
    \begin{split}
    -\eps \Delta \widetilde{u} + \frac{1}{\eps}\left(\widetilde{u}^2 - 1 - \sqrt{2}\beta \right)\widetilde{u} = &\left( 1- \frac{1}{\lambda_\eps^2} \right)\frac{1}{\eps}\left(\widetilde{u}^2 - 1 - \sqrt{2}\beta \right)\widetilde{u} \\
    &+ \sqrt{2}\beta \lambda_\eps \left( \frac{\rho-1}{\eps} - \kappa_\star \right) \widetilde{u} \\
    &+ \lambda_\eps^2 f_\eps(u_1,\,u_2).
    \end{split}
	\end{equation}
	By a straightforward computation, we have
	\[
		\frac{1}{\eps} \left( 1- \frac{1}{\lambda_\eps^2} \right)^2 \int_B \frac{1}{\eps^2}\left(\widetilde{u}^2 - 1 - \sqrt{2}\beta \right)^2\widetilde{u}^2 \,{\d}x = \frac{\beta^2}{8} \int_B \frac{1}{\eps}\left(\widetilde{u}^2 - 1 - \sqrt{2}\beta \right)^2\widetilde{u}^2 \,{\d}x .
	\]	
	Recalling that $\abs{\u}^2$ is uniformly bounded depending only on $\beta$, 
	that $\lambda_\eps \to 1$ as $\eps \to 0$, and that $\kappa_\star = \frac{\beta}{2\sqrt{2}}\left( 1 + \sqrt{2}\beta \right)$, 
	we see that
	\begin{equation}\label{eq:decay-widetilde-u}
		\frac{1}{\eps} \left( 1- \frac{1}{\lambda_\eps^2} \right)^2 \int_B \frac{1}{\eps^2}\left(\widetilde{u}^2 - 1 - \sqrt{2}\beta \right)^2\widetilde{u}^2 \,{\d}x \leq C(\beta) \int_B \frac{1}{\eps}h(\u)\,{\d}x \stackrel{\eqref{eq:h-bis}}{\leq} C_\beta(x_0,\,R),
	\end{equation} 
	where the constant $C(\beta)$ depends only on $\beta$ and $C_\beta(x_0,\,R)$ 
	depends only on $\beta$, $\gamma$, $x_0$, $R$, $\Omega$, $\Cpot$, 
    and the $L^1(\partial \Omega)$- and the $L^2(\partial \Omega)$-norm 
    of $\Qb \times \partial_\ttau \Qb$.
	Upon setting
	\begin{equation}\label{eq:capitolF}
	%\begin{split}
		F_\eps\left( \widetilde{u},\,u_2\right) 
		:=  \frac{\beta}{2\sqrt{2}}\left(\widetilde{u}^2 - 1 - \sqrt{2}\beta \right)\widetilde{u} 
		+ \sqrt{2}\beta\lambda_\eps \left( \frac{\rho-1}{\eps} - \kappa_\star \right) \widetilde{u} 
		 + f_\eps(u_1,\,u_2).
	%\end{split}
	\end{equation}
	from~\eqref{eq:decay-f},~\eqref{eq:decay-rho-pot}, and~\eqref{eq:decay-widetilde-u} 
	we see that
	\begin{equation}\label{eq:decay-capitolF-local}
		\frac{1}{\eps} \int_B F_\eps^2\,{\d}x \leq C_\beta(x_0,\,R),
	\end{equation}
	where the constant $C_\beta(x_0,\,R)$ depends only on 
	$\beta$, $\gamma$, $x_0$, $R$, $\Omega$, $\Cpot$, 
    and the $L^1(\partial \Omega)$- and the $L^2(\partial \Omega)$-norm 
    of $\Qb \times \partial_\ttau \Qb$. 
    Finally, let $K \subset \Omega \setminus \spt\mu_\star$ be any compact set. 
	Then, by~\eqref{eq:decay-capitolF-local} and a standard covering argument, we have 
	\[
		\frac{1}{\eps}\int_K F_\eps^2\,{\d}x \leq C_{\beta,K} 
	\]
	where the constant $C_{\beta,K}$, depends only on $\beta$, $K$, 
	$\Omega$, $\Cpot$, and the $L^1(\partial \Omega)$- and the 
	$L^2(\partial \Omega)$-norm of $\Qb \times \partial_\ttau \Qb$. 
	This proves~\eqref{eq:decay-capitolF}. 
	Combining~\eqref{eq:u1-pert-ter} and~\eqref{eq:capitolF}, 
	we obtain~\eqref{eq:EL-u1-bis}. This concludes 
	the proof of the claim.   
	\end{step}
	
	\begin{step}[Inferring integrality from \cite{NagaseTonegawa}, locally in $\Omega \setminus \spt\mu_\star$]
	Let $K \subset \Omega \setminus \spt \mu_\star$ be any compact set. 
	We observe that Equation~\eqref{eq:EL-u1-bis} has the same form 
	as \cite[Equation~(1.2)]{NagaseTonegawa} with $F_\eps$ that, as a function 
	of $x \in K$, satisfies \cite[Assumption~(A.2)]{NagaseTonegawa}. 
	By the definition of $\widetilde{u}_{1,\eps}$, the energy 
	bound \cite[(A.1)]{NagaseTonegawa} is satisfied because 
	of~\eqref{eq:ACeps-equibdd} and~\eqref{eq:tildeetaeps-etaeps-pointwise} below.   
	We consider the energy densities 
	(since at this point the component $u_2$ has been ruled out, we further 
	simplify the notation writing $u$ in place of $u_1$)
	\[
		\widetilde{\eta}_\eps := 
		\frac{\eps\abs{\nabla \widetilde{u}_{\eps}}^2}{2} 
		+ \frac{\left(\widetilde{u}_\eps^2-1-\sqrt{2}\beta\right)^2}{4 \eps},  
	\]
	as well as 
	\[
		\eta_\eps := \frac{\eps\abs{\nabla u_{\eps}}^2}{2} 
		+ \frac{\left(u_\eps^2-1-\sqrt{2}\beta\right)^2}{4 \eps}.
	\]
	In view of the definition~\eqref{eq:widetilde-u} of $\widetilde{u}_\eps$, 
	for every $\eps > 0$, we have 
	\begin{equation}\label{eq:tilde-zeta-h1-eps-zeta-h1-eps}
	\begin{split}
		\widetilde{\zeta}^{(h_1)}_\eps - \zeta^{h_1}_\eps &= 
		\frac{1}{\eps}( h_1(\widetilde{u}_\eps) - h_1(u_\eps) ) \\
		&= \frac{1}{4\eps} \left\{ \left(\lambda_\eps^2 u^2_\eps - 1 - \sqrt{2}\beta \right)^2 - \left(u^2_\eps - 1 - \sqrt{2}\beta \right)^2 \right\} \\
		&= \frac{u_\eps^2}{4\eps}\left\{ \left(\lambda_\eps^2 - 1\right)^2 u_\eps^2 + 2\left(1-\lambda_\eps^2\right)^2 \left(u_\eps^2-1-\sqrt{2}\beta\right)  \right\},
	\end{split}
	\end{equation}
	so that, by the definition of $\lambda_\eps$ in~\eqref{eq:widetilde-u} 
	and~\eqref{eq:ACeps-equibdd},
	\begin{equation}\label{eq:tilde-zeta-h1-eps-zeta-h1-eps-limit}
		\int_K \abs{\widetilde{\zeta}^{(h_1)}_\eps - \zeta^{(h_1)}_\eps}\,{\d}x 
		\to 0, \qquad \mbox{as } \eps \to 0.
	\end{equation}
	On the other hand, 
	\begin{equation}\label{eq:tilde-nabla-u-nabla-u}
		\abs{\nabla \widetilde{u}_\eps}^2 - \abs{\nabla u_\eps}^2 
		= \left( \lambda_\eps^2 - 1 \right) \abs{\nabla u_\eps}^2,
	\end{equation}
	whence
	\begin{equation}\label{eq:tilde-nabla-u-nabla-u-limit}
		\int_K \left( \abs{\nabla \widetilde{u}_\eps}^2 - \abs{\nabla u_\eps}^2\right)\,{\d}x 
		\to 0, \qquad \mbox{as } \eps \to 0.
	\end{equation}
	From~\eqref{eq:tilde-zeta-h1-eps-zeta-h1-eps-limit} 
	and~\eqref{eq:tilde-nabla-u-nabla-u-limit} it follows that 
	\begin{equation}\label{eq:tildeetaeps-etaeps-pointwise}		
	\widetilde{\eta}_\eps - \eta_\eps 
		= \eps\lambda\left(\lambda_\eps^2-1\right)\abs{\nabla u_\eps}^2 
		- \frac{u^2_\eps}{4\eps}\left\{ \left(\lambda_\eps^2 - 1\right)^2 u_\eps^2 + 2\left(1-\lambda_\eps^2\right)^2 \left(u_\eps^2-1-\sqrt{2}\beta\right)  \right\}, 
	\end{equation}
	and thus 
	\begin{equation}\label{eq:tildeetaeps-etaeps}
		\int_K \abs{\widetilde{\eta}_\eps - \eta_\eps} \,{\d}x \to 0, 
		\qquad \mbox{as } \eps \to 0.
	\end{equation}
	Since the sequences of positive functions 
	$\left( \widetilde{\eta}_\eps \right)_{\eps > 0}$, 
	$\left( \eta_\eps \right)_{\eps > 0}$ are bounded in $L^1(K)$, 
	for every $K \subset \Omega \setminus \spt \mu_\star$, by standard compactness 
	properties of Radon measures, they 
	converge in the sense of Radon measures along a subsequence to limiting measures 
	$\widetilde{\eta}_\star$, $\eta_\star$ defined on the whole $\Omega \setminus \spt\mu_\star$. 
	In view of~\eqref{eq:tildeetaeps-etaeps}, we have 
	\begin{equation}\label{eq:tildeeta*-eta*}
		\eta_\star \mres K = \widetilde{\eta}_\star \mres K	
	\end{equation}
	independently of the chosen subsequence.
	
	Theorem~4.1 in \cite{NagaseTonegawa} gives that $\widetilde{\eta}_\star$ is 
	$\H^1$-rectifiable, i.e., 
	$\widetilde{\eta}_\star = \widetilde{\theta}_\star \H^1 \mres \Sigma$, 
	where $\Sigma$ is $\H^1$-rectifiable and $\widetilde{\theta}_\star$ 
	denotes the density of $\widetilde{\eta}_\star$. In particular, the limit  
	\[
		\widetilde{\theta}_\star(x_0) := 
		\lim_{r \to 0} \frac{\widetilde{\eta}_\star(B(x_0,\,r))}{2r}
	\] 
	exists at any point $x_0 \in \Sigma$. Since rectifiability 
	is a local property, the rectifiability of $\widetilde{\eta}_\star$ 
	and~\eqref{eq:tildeeta*-eta*} imply 
	that $\eta_\star$ is rectifiable as well and, moreover, that  
	\[
		\theta_\star(x_0) = \widetilde{\theta}_\star(x_0), 
		\qquad \forall x_0 \in K.
	\]
	Furthermore, if we set
	\[
		\sigma_\beta := 
		\int_{-\sqrt{1+\sqrt{2}\beta}}^{\sqrt{1+\sqrt{2}\beta}} \sqrt{\frac{h_1(s)}{2}}\,{\d}s = \frac{\sqrt{2}}{3}\left(1 + \sqrt{2}\beta\right)^{3/2}, 
	\]
	then applying \cite[Theorem~5.1]{NagaseTonegawa} and 
	using~\eqref{eq:tilde-zeta-h1-eps-zeta-h1-eps} 
	and~\eqref{eq:tilde-nabla-u-nabla-u} leads to the fact that 
	$(2 \sigma_\beta)^{-1}{\theta_\star(x_0)}$ is an integer for all $x_0 \in K$. 
	In addition, thanks to~\eqref{eq:tilde-zeta-h1-eps-zeta-h1-eps-limit} 
	and~\eqref{eq:tilde-nabla-u-nabla-u}, \cite[Proposition~4.3]{NagaseTonegawa} tells 
	us that the discrepancy functions 
	\[
		\xi_\eps^{(h_1)} :=  
		\frac{1}{\eps} h_1\left(u_\eps\right) 
		- \frac{\eps}{2}\abs{\nabla u_\eps}^2
	\]
	satisfy
	\[
		\lim_{\eps \to 0} \int_K \abs{\xi^{(h_1)}_\eps} \,{\d}x = 0,
	\]
	i.e., equipartition holds in the limit. More precisely, if we define 
	the density of the potential energy functions (dually seen as measures) 
	\begin{equation}\label{eq:def-tilde-zeta-eps-h1}
		{\zeta}_{\eps}^{(h_1)}
		:= \frac{1}{\eps} h_1\left({u}_\eps\right), 
	\end{equation}
	we see that there exists a limiting potential energy measure 
	${\zeta}^{(h_1)}_\star$, which satisfies
	\[
		{\eta}_\star \mres K = 2 {\zeta}_\star^{(h_1)} \mres K,  
	\]
	so that, in particular, 
	\[
		\spt \left( {\eta}_\star \mres K \right) 
		= \spt \left( {\zeta}^{(h_1)}_\star \mres K \right).
	\] 
	Moreover, the density ${\mathfrak{h}}_\star(x_0)$ 
	of ${\zeta}_\star^{(h_1)}$ 
	exists at every point $x_0 \in K$ and satisfies  
	\[
	\lim_{r \to 0} \frac{{\zeta}^{(h_1)}_\star(B(x_0,\,r))}{2r} 
	= {\mathfrak{h}}_\star(x_0) = \frac{{\theta}_\star(x_0)}{2},
	\]
	so that $\sigma_\beta^{-1} {\mathfrak{h}}_\star(x_0)$ is an integer, 
	for any $x_0 \in K$. 
	Finally, by letting $K$ vary in 
	$\Omega \setminus \spt\mu_\star$, we obtain
	\begin{equation}\label{eq:h*-integer}
		\sigma_\beta^{-1} \mathfrak{h}_\star(x_0) \in \mathbb{N}, 
		\qquad \forall x_0 \in \Omega \setminus \spt\mu_\star. 
	\end{equation}
	\end{step}
	
	\begin{step}[Conclusion]
	In order to conclude the proof, we have to show that the 
	integrality of the rescaled density $\sigma_\beta^{-1} \mathfrak{h}_\star$  
	implies that the density $\mathfrak{v}_\star$ of $\zeta_\star$ 
	is integer-valued as well after rescaling, i.e, that 
	$\sigma_\beta^{-1} \mathfrak{v}_\star$ takes only integer values. 
	In fact, we are going to prove that 
	\begin{equation}\label{eq:v*=h*}
		\mathfrak{v}_\star(x_0) = \mathfrak{h}_\star(x_0), 
		\qquad \forall x_0 \in \Omega \setminus \spt\mu_\star.
	\end{equation}
	To this purpose, it suffices to 
	show that the difference $\zeta_\eps - \widetilde{\zeta}^{(h_1)}_\eps$ 
	tends to zero in $L^1\left(K\right)$ as $\eps \to 0$.
	%, for any  
	%simply connected, open set $G \csubset \Omega \setminus \spt\mu_\star$ 
	%with smooth boundary $\partial G$. (Being the conclusion pointwise 
	%in nature, it will follow by letting $G$ vary in $\Omega \setminus \spt \mu_\star$.)
	By the definition~\eqref{eq:zeta-eps} of $\zeta_\eps$, the 
	definition~\eqref{eq:def-tilde-zeta-eps-h1} of 
	${\zeta}^{(h_1)}_\eps$,~\eqref{eq:V-h}, and~\eqref{eq:h-bis}, we have 
	\begin{equation}\label{eq:diff-potentials}
	\begin{split}
		\zeta_\eps - \zeta^{(h_1)}_\eps 
		&= \frac{1}{\eps}\left( h(\u) + \frac{\beta}{\sqrt{2}} (1-\rho)\left(u_1^2 - u_2^2 - 1 - \frac{\beta + \beta \rho}{\sqrt{2}} \right) \right) 
		- \frac{1}{\eps}h_1(u_{1}) \\
		&= \frac{1}{\eps} h_2(u_1,\,u_2) + \frac{\beta}{\sqrt{2}} \left(\frac{1-\rho}{\eps}\right)\left(u_1^2 - 1 - \sqrt{2}\beta\right) + \frac{\beta}{\sqrt{2}}\left( \frac{1-\rho}{\eps} \right) \left( \frac{\beta}{\sqrt{2}}(\rho - 1) - u_2^2 \right)
	\end{split}
	\end{equation}
	All terms on the right-hand side of~\eqref{eq:diff-potentials} 
	vanish in $L^1\left( K \right)$ as $\eps \to 0$. 
	Indeed, by~\eqref{eq:decay-h2} (recall that $h_2(u_1,\,u_2) \geq 0$ pointwise, 
	if $\sqrt{2}\beta \geq 1$) and a covering argument, 
	\[
		\frac{1}{\eps} \int_K h_2(u_1,\,u_2) \,{\d}x \lesssim \eps^\alpha, 
		\qquad \mbox{as } \eps \to 0, 
	\]
	for every $\alpha \in (0,\,2)$. Next, by H\"{o}lder's inequality 
	and Lemma~\ref{lemma:improved-decay-rho}, 
	\[
	\begin{split}
		\int_K \abs{\left(\frac{1-\rho}{\eps}\right)\left(u_1^2 - 1 - \sqrt{2}\beta\right)} 
		&\leq \left(\int_K \left(\frac{1-\rho}{\eps}\right)^2 \,{\d}x \right)^{1/2} \left( \int_K \left(u_1^2 - 1 - \sqrt{2}\beta\right)^2 \,{\d}x \right)^{1/2} \\ 
		&\lesssim \sqrt{\eps}, \qquad \mbox{as } \eps \to 0.
	\end{split}
	\]
	Finally, by \cite[Proposition~2.6]{CDS1} and 
	Proposition~\ref{prop:decay-u2}, 
	\[
		\frac{1}{\eps} \int_K \left\{ (\rho - 1)^2 + u_2^2 \right\}\,{\d}x 
		\lesssim \eps, \qquad \mbox{as } \eps \to 0.
	\]
	Combining the last three inequalities with~\eqref{eq:diff-potentials}, 
	it follows that
	\[
		\int_K \abs{ \zeta_\eps - \zeta_\eps^{(h_1)} }\,{\d}x 
		\lesssim \sqrt{\eps}, \qquad \mbox{as } \eps \to 0.
	\]
	As a consequence, 
	\[
		\zeta_\star \mres K = \zeta_\star^{(h_1)} \mres K.
	\]
	In particular,
	\[
		\mathfrak{v}_\star = \mathfrak{h}_\star \qquad \mbox{on } K.
	\]
	Since $K$ was arbitrary in $\Omega \setminus \spt\mu_\star$, 
	by letting $K$ vary in $\Omega \setminus \spt\mu_\star$, we have 
	\[
		\mathfrak{v}_\star(x_0) = \mathfrak{h}_\star(x_0),  
		\qquad \forall x_0 \in \Omega \setminus \spt \mu_\star,
	\]
	i.e., we have proved~\eqref{eq:v*=h*}. 
	In particular, by~\eqref{eq:v*=h*} and~\eqref{eq:h*-integer}, 
	it follows that $\sigma_\beta^{-1} \mathfrak{v}_\star$ is an integer-valued 
	function on $\Omega \setminus \spt\mu_\star$. 
	Finally, the conclusion follows because, by \cite[Proposition~3.18(ii)]{CDS2},  
	$\mathfrak{v}_\star(a) = 0$ for every $a \in \spt\mu_\star$. 
	\end{step}
\end{proof}

%--------------------------
\paragraph{Acknowledgments}
F.L.D. would like to thank Institute of Science Tokyo for warm hospitality.
F.L.D. is a member of GNAMPA-INdAM and he is partially supported by 
INdAM-GNAMPA Project CUP E53C25002010001. Y.T. is partially supported by 
JSPS Grant-in-aid for scientific research (A) {\#}23H00085. 

\bibliographystyle{plain}
\bibliography{biblio}

\begin{flushright}
\Addresses
\end{flushright}

\end{document}